\documentclass[reqno]{amsart}

\usepackage{amsmath}
\usepackage{amsfonts}
\usepackage{xcolor}
\usepackage{color}
\usepackage{graphicx}
\usepackage{amssymb}
\usepackage{soul}
\usepackage{graphicx}
\usepackage{amsthm}
\usepackage{fullpage}
\usepackage{paralist}
\usepackage[all]{xy}
\usepackage{tikz}
\usepackage{verbatim}
\usepackage{cite}
\usepackage{mathrsfs}
\usepackage{bm}

\usepackage{hyperref}

\theoremstyle{plain}

\newtheorem{thm}{Theorem}[section]
\newtheorem{lem}[thm]{Lemma}

\newtheorem{prop}[thm]{Proposition}

\theoremstyle{definition}

\newtheorem{defn}[thm]{Definition}

\newtheorem{ex}[thm]{Example}
\newtheorem{rem}[thm]{Remark}
\newtheorem{prob}[thm]{Problem}

\theoremstyle{rem}

\newcommand{\N}{\mathbb{N}}
\newcommand{\Z}{\mathbb{Z}}
\newcommand{\cD}{\mathcal{D}}
\newcommand{\cG}{\mathcal{G}}
\newcommand{\X}{\mathfrak{X}}
\newcommand{\e}{\epsilon}
\newcommand{\U}{\Upsilon}
\newcommand{\tree}{\mathcal{T}}
\DeclareMathOperator{\Aut}{\rm Aut}
\DeclareMathOperator{\Sym}{\rm Sym}
\DeclareMathOperator{\Alt}{\rm Alt}

\DeclareMathOperator{\id}{id}
\DeclareMathOperator{\Iso}{\rm Iso}
\DeclareMathOperator{\pr}{\rm pr}

\numberwithin{equation}{section}

\usepackage{color}
\definecolor{darkgreen}{cmyk}{1,0,1,.2}
\definecolor{m}{rgb}{1,0.1,1}

\newdimen\theight
\def\TeXref#1{%
             \leavevmode\vadjust{\setbox0=\hbox{{\tt
                     \quad\quad  {\small \textrm #1}}}%
             \theight=\ht0
             \advance\theight by \lineskip
             \kern -\theight \vbox to
             \theight{\rightline{\rlap{\box0}}%
             \vss}%
             }}%

\begin{document}

\begin{abstract}
The discriminant group of a minimal equicontinuous action of a group $G$ on a Cantor set $X$ is the subgroup of the closure of the action in the group of homeomorphisms of $X$, consisting of homeomorphisms which fix a given point. The stabilizer and the centralizer groups associated to the action are obtained as direct limits of sequences of subgroups of the discriminant group with certain properties. 
Minimal equicontinuous group actions on Cantor sets admit a classification by the properties of the stabilizer and centralizer direct limit groups.

In this paper, we construct new families of examples of minimal equicontinuous actions on Cantor sets, which illustrate certain aspects of this classification. These examples are constructed as actions on rooted trees. The acting groups are countable subgroups of the product or of the wreath product of groups. We discuss applications of our results to the study of attractors of dynamical systems and of minimal sets of foliations.

\end{abstract}
	
\title{Wild Cantor actions}

\author{Jes\'us \'Alvarez L\'opez}
\address{Jes\'us \'Alvarez L\'opez, Department of Mathematics, 
         Faculty of Mathematics,
         University of Santiago de Compostela,
         15782 Santiago de Compostela, Spain}
\email{jesus.alvarez@usc.es}

\author{Ramon Barral Lijo}
\address{Ramon Barral Lijo, Research Organization of Science and Technology, Ritsumeikan University, 1 Chome-1-1 Nojihigashi, Kusatsu, Shiga 525-8577, Japan}
\email{ramonbarrallijo@gmail.com}

\author{Olga Lukina}
\address{Olga Lukina, Faculty of Mathematics, University of Vienna, Oskar-Morgenstern-Platz 1, 1090 Vienna, Austria}
\email{olga.lukina@univie.ac.at}

\author{Hiraku Nozawa}
\address{Hiraku Nozawa, Department of Mathematical Sciences, College of Science and Engineering, Ritsumeikan University, 1 Chome-1-1 Nojihigashi, Kusatsu, Shiga 525-8577, Japan}
\email{hnozawa@fc.ritsumei.ac.jp}

\thanks{Version date: November 24, 2020}

\thanks{JAL is partially supported by Project MTM2017-89686-P (AEI/FEDER, UE); RBL is partially supported by a Canon Foundation in Europe Research Fellowship; OL is supported by the FWF Project P31950-N35; HN is partially supported by JSPS KAKENHI Grant number 17K14195 and 	20K03620}

\thanks{MSC 2020: Primary 37B05, 37E25, 20E08, 20E15, 20E18, 20E22, 22F05, 22F50. Secondary: 20F22, 57R30, 57R50.}

\thanks{Keywords: group actions, Cantor sets, equicontinuous actions, group actions on rooted trees, wreath products, profinite groups, stabilizer direct limit group, centralizer direct limit group, the alternating group, the cyclic group}

\maketitle

\section{Introduction}

Let $X$ be a Cantor set, that is, a compact totally disconnected perfect metric space with metric $D$, and let $G$ be a finitely generated group. Let $\Phi: G \to \operatorname{Homeo}(X)$ be an action of $G$ on $X$. We also denote such an action by $(X,G,\Phi)$, and use the shortcut $g \cdot x:=\Phi(g)(x)$.
We assume that the action $(X,G,\Phi)$ is minimal, that is, for every $x \in X$ the orbit $G \cdot x = \{ g \cdot x \mid g \in G\}$ is dense in $X$, and that $(X,G,\Phi)$ is \emph{equicontinuous}, that is, for every $\e>0$ there is $\delta>0$ such that for every $g \in G$ and every $x,y \in X$ with $D(x,y)< \delta$ we have $D(g \cdot x, g \cdot y)< \e$.

Examples of equicontinuous actions on Cantor sets are abundant in mathematics. For instance, actions of self-similar groups on the boundaries of regular rooted trees are equicontinuous \cite{BN2008,Grig2011,Nekrashevych2005}. Also, arboreal representations of absolute Galois groups of number fields give rise to such actions \cite{Lukina2019}, see \cite{Jones2013} for a further discussion of arboreal representations. In continuum dynamics, equicontinuous actions on Cantor sets arise as actions of the fundamental group of a closed manifold on the fibre of a weak solenoid over this manifold \cite{McCord1965,FO2002}. A weak solenoid is the inverse limit of an infinite sequence of finite-to-one covering spaces of a closed manifold $M$. A weak solenoid is a fibre bundle over $M$ with a Cantor set fibre, so it is a flat bundle \cite{CC2000} and an example of a foliated space with totally disconnected transversals.

In foliation theory, foliated spaces which are transversely Cantor sets, also called laminations, arise as minimal sets of smooth foliated manifolds. Very little is known about minimal sets of foliations of codimension and leaf dimension $2$ and higher. It was shown in \cite{ClarkHurder2013} that a lamination with equicontinuous transverse dynamics is homeomorphic to a weak solenoid, and so the holonomy pseudogroup of such a lamination is induced by a group action. Embeddings of certain weak solenoids over tori as minimal sets of foliations were constructed in \cite{ClarkHurder2011}, and some sufficient conditions for a lamination to admit an embedding as a minimal set were studied in \cite{Hurder2017}. Even when a lamination has  equicontinuous dynamics, it is in general not known when it admits an embedding as a minimal set of a foliation, except for the toral case discussed above. One of the reasons for that may be that Cantor sets, being totally disconnected, are extremely flexible, and the actions of non-abelian groups on Cantor sets exhibit a wide variety of phenomena, which are still not fully understood. To answer the embedding question, one must better understand the behavior of the actions considered, and this is one of the motivations of our study of the classification of equicontinuous group actions on Cantor sets. Some, but most likely not all of the phenomena exhibited by minimal equicontinuous actions on Cantor sets can also be observed for equicontinuous minimal actions on connected topological spaces \cite{ALC2009,ALM2016,AB2019}.

Algebraic invariants which classify equicontinuous minimal Cantor actions were introduced in the works of the third author joint with Dyer and Hurder \cite{DHL2017,HL2019a,HL2019b}. We describe this classification in more detail below. A natural question is then whether all possibilities given by the classification are realised in actual examples. Papers \cite{DHL2017,HL2019a,HL2019b} provide constructions of several classes of actions illustrating some of the aspects of this classification, but the examples constructed do not capture the more subtle aspects of the classification scheme. In this paper, we give the constructions of new classes of equicontinuous minimal Cantor actions, using techniques from theory of group actions on rooted trees, to show that all possibilities occurring in the classification in \cite{DHL2017,HL2019a,HL2019b} are realised.

While our methods for constructing actions are standard, it is the careful choices made in the constructions that are critical for the resulting dynamical properties of the actions. These choices sometimes reveal unexpected connections, for instance, with number-theoretical problems. For example in Remark~\ref{rem-zhang} we observe that the solution of the bounded gaps conjecture by Zhang can be applied to obtain many new classes of actions on rooted trees with exotic dynamical properties. In Example \ref{ex-arith} we observe that action of wreath products of cyclic groups, which are among the family of actions studied in Theorem \ref{thm-main2}, arise in number theory and arithmetic dynamics as arboreal representations of absolute Galois groups of number fields into the automorphism groups of rooted trees.

\medskip
We now give a brief sketch of the classification by algebraic invariants in order to state our main results.

Let $(X,G,\Phi)$ be a minimal equicontinuous action of a countably generated group $G$ on a Cantor set $X$. Throughout the paper we assume that the action of $G$ on $X$ is effective, that is, $g \cdot x = x$ for all $x \in X$ implies $g = \id$.
Then the closure $E(G) := \overline{\Phi(G)} \subset \operatorname{Homeo}(X)$ in the uniform topology is a profinite group identified with the Ellis group of the action $(X,G,\Phi)$ \cite{Auslander1988,Ellis1960,EllisGottschalk1960}. Elements of $E(G)$ can be thought of as sequences $\widehat{g} = (g_i)$, $g_i \in G$ for $i \geq 0$, and the image  $\Phi(G)$ is a dense subgroup of $E(G)$. If the action $(X,G,\Phi)$ is effective, then $\Phi(G)$ is identified with $G$. In any case, the action $\Phi$    induces an action $\widehat{\Phi}:E(G) \to \operatorname{Homeo}(X)$, which is transitive since $(X,G,\Phi)$ is assumed to be minimal. 

Let $x$ be a point in $X$, and denote by 
  \begin{align*}E(G)_x = \{\widehat{g} \in E(G) \mid \widehat{g} \cdot x =x\}\end{align*}
the isotropy group of the action of $E(G)$ at $x \in X$. Then $X = E(G)/E(G)_x$ is a homogeneous space for the action of $E(G)$.

Since $E(G)$ is a profinite group, its identity element admits a fundamental system $\{\widehat{V}_\ell\}_{\ell \geq 0}$, $\widehat{V}_0 = E(G)$, of open neighborhoods, such that each $\widehat V_\ell$ is a normal subgroup of $E(G)$ of finite index \cite[Theorem 2.1.3]{RZ}. Then for each $\ell \geq 0$ the subgroup $\widehat V_\ell$ is a clopen normal subgroup of $E(G)$, and the product $\widehat{U}_\ell = E(G)_x \, \widehat{V}_\ell$ is a clopen subgroup of $E(G)$. Recall that  a clopen subset $V \subset X$ is a closed and open subset. The quotient space $U_\ell = \widehat{U}_\ell/E(G)_x$ is a clopen neighborhood of $x$ in $X$. The isotropy group $E(G)_x$ is a normal subgroup of $E(G)$ if and only if it is trivial \cite{DHL2016}, so if $E(G)_x$ is non-trivial, then $U_\ell$ is not a group. We define an increasing chain of subgroups of $E(G)_x$ by
  \begin{align}\label{eq-Kchain} K(\Phi) = \{K_\ell\}_{ \ell \geq 0}, &  &K_\ell =  \left\{ \widehat{h} \in E(G)_x \mid \widehat{h} \, \widehat{g} \, E(G)_x = \widehat{g} \, E(G)_x, \, \textrm{ for all }\widehat{g} \in \widehat{U}_\ell \right\},\end{align}
that is, $K_\ell$ consists of all elements in $E(G)_x$ which act as the identity on the clopen set $U_\ell$. Next, define the adjoint action of the isotropy group $E(G)_x$ on $E(G)$ by 
  $$Ad(\widehat{h})(\widehat{g}) = \widehat{h} \, \widehat{g} \, \widehat{h}^{-1} \textrm{, for } \widehat{h} \in E(G)_x \textrm{ and } \widehat{g} \in E(G).$$ 
  
We have $E(G)_x = \bigcap_{\ell \geq 0} \widehat{U}_\ell$ \cite{HL2019b}. Since $E(G)_x \subset \widehat U_\ell$, the adjoint action of $E(G)_x$ preserves the clopen subgroups $\widehat{U}_\ell$ for $\ell \geq 0$. Thus we can define an increasing chain of subgroups by 
 \begin{align}\label{eq-Zchain} Z(\Phi) = \{Z_\ell\}_{ \ell \geq 0}, & & Z_\ell = \left\{ \widehat{h} \in E(G)_x \mid Ad(\widehat{h})(\widehat{g}) = \widehat{g}, \, \textrm{ for all }\widehat{g} \in \widehat{U}_\ell \right\}.\end{align}
We have that $Z_\ell \subset K_\ell$ for $\ell \geq 0$. 

For $\ell \geq 0$, denote by $\iota_\ell^{\ell+1}: K_\ell \to K_{\ell+1}$ the inclusions maps, and note that the maps $\iota_\ell^{\ell+1}$ restrict to inclusions $\iota_\ell^{\ell+1}: Z_\ell 
\to Z_{\ell+1}$. Thus we obtain two directed systems of groups, $\cG(K_\ell, \iota_\ell^{\ell+1}, \N)$ and $\cG(Z_\ell, \iota_\ell^{\ell+1}, \N)$, over the totally ordered set of natural numbers. For $k>\ell$, denote by $\iota^k_\ell = \iota^{k}_{k-1} \circ \cdots \circ \iota^{\ell+1}_\ell$ the composition of the inclusions maps. Form the direct limit groups for these systems as usual, that is, 
  \begin{align}\label{eq-directlimk} \Upsilon^x_s(\Phi) =\lim_{\longrightarrow} 
\cG(K_\ell, \iota_\ell^{\ell+1}, \N) = \bigcup_{\ell \in \N} K_\ell \Big\slash \left( \widehat{h}_\ell \sim \widehat{h}_{\ell'} \right),
 \end{align}
where $\widehat{h}_\ell \sim \widehat{h}_{\ell'} $ if and only if there is $k > \ell, \ell'$ such that $\iota^k_\ell(\widehat{h}_\ell) = \iota^k_{\ell'}(\widehat{h}_{\ell'})$, and similarly
\begin{align}\label{eq-directlimz} \Upsilon^x_c(\Phi) =\lim_{\longrightarrow} 
\cG(Z_\ell, \iota_\ell^{\ell+1}, \N) = \bigcup_{\ell \in \N} Z_\ell \Big\slash \left( \widehat{h}_\ell \sim \widehat{h}_{\ell'} \right).
 \end{align}
The groups \eqref{eq-directlimk}  and \eqref{eq-directlimz} are called the \emph{stabilizer} and the \emph{centralizer (direct limit) groups} respectively \cite{HL2019b}. Note that we have the canonical inclusion
\begin{equation}\label{eq-cs}
    \upsilon_x : \Upsilon^x_c(\Phi) \longrightarrow \Upsilon^x_s(\Phi).
\end{equation}
It was proved in \cite{HL2019b} that although the groups $K_\ell$ appearing in \eqref{eq-directlimk}  and  $Z_\ell$ appearing in \eqref{eq-directlimz}  depend on the choice of a system of neighborhoods of the identity $\{\widehat{V}_\ell\}_{\ell \geq 0}$ and on the choice of a point $x \in X$, the isomorphism classes of the direct limits do not depend on these choices. More precisely, we have:

\begin{thm}[{\cite[Theorem 4.15]{HL2019b}}]
Let $(X,G,\Phi)$ be a minimal equicontinuous action of a finitely generated group $G$ on a Cantor set $X$. Then the direct limit isomorphism classes $\U_s(\Phi)$ and $\U_c(\Phi)$ of the groups $\U^x_s(\Phi)$ and $\U^x_c(\Phi)$ are invariants of the conjugacy class of the action $(X,G,\Phi)$.
\end{thm}

Now we can introduce a classification of minimal equicontinuous group Cantor actions as in \cite{HL2019b}. We say that the directed system of groups $\cG(K_\ell, \iota_\ell^{\ell+1}, \N)$ (resp. $\cG(Z_\ell, \iota_\ell^{\ell+1}, \N)$) is \emph{bounded} if it has a maximal element, that is, there exists $n \geq 0$ such that for any $\ell \geq n$ the inclusion map $\iota_\ell^{\ell+1}: K_\ell \to K_{\ell+1}$ (resp. $\iota_\ell^{\ell+1}: Z_\ell \to Z_{\ell+1}$) is an isomorphism. An isomorphism class $\U_s(\Phi)$ (resp. $\U_c(\Phi)$) of direct limit groups 
is \emph{bounded} if $\U_s^x(\Phi)$ (resp. $\U_c^x(\Phi)$) is represented by a bounded group chain.

\begin{defn}\label{defn-classif}
Let $(X,G,\Phi)$ be a minimal equicontinuous action, and let $(X,E(G),\widehat{\Phi})$ be the action of its Ellis group. The action $(X,G,\Phi)$ or $(X,E(G),\widehat{\Phi})$ is said to be:
\begin{enumerate}
\item \emph{stable} if the stabilizer group $\U_s(\Phi)$ is bounded, and \emph{wild} otherwise,
\item \emph{algebraically stable} if the centralizer group $\U_c(\Phi)$ is bounded, and \emph{algebraically wild} otherwise,
\item  \emph{wild of finite type} if $\U_s(\Phi)$ is unbounded and represented by a chain of finite subgroups, 
\item \emph{algebraically wild of algebraic finite type} if $\U_c(\Phi)$ is unbounded and represented by a chain of finite subgroups,
\item \emph{wild of flat type} if $\U_s(\Phi)$ is unbounded and $\U_c(\Phi) = \U_s(\Phi)$, namely, the canonical inclusion $\Upsilon^x_c(\Phi) \longrightarrow \Upsilon^x_s(\Phi)$ in \eqref{eq-cs} is an isomorphism for some, and so any $x \in X$. 
\item \emph{dynamically wild} if $\U_s(\Phi)$ is unbounded and not of flat type. 
\end{enumerate} 
\end{defn}

All properties listed in Definition \ref{defn-classif} are invariants of the conjugacy class of a given action $(X,G,\Phi)$ \cite{HL2019b}. 

We require one more concept before stating the classification of actions, which is the definition of a \emph{non-Hausdorff element} in the group $G$ or $E(G)$. The notion of a non-Hausdorff element comes from the study of germinal groupoids associated to an action $(X,G,\Phi)$. In particular, such a groupoid has non-Hausdorff topology if and only if it contains a non-Hausdorff element \cite{Winkelnkemper1983}. No further knowledge of groupoids is required for the rest of the article.

\begin{defn}\label{defn-nonHausdorff}
Let $(X,G,\Phi)$ be an equicontinuous minimal action, and let $E(G)$ be the Ellis group of the action. An element $\widehat{g} \in E(G)$ is \emph{non-Hausdorff} if the following conditions are satisfied at the same time: 
\begin{enumerate}
\item There is $x \in X$ such that $\widehat{g} \cdot x = x$.
\item There is a decreasing sequence of open neighborhoods $\{U_\ell\}_{\ell \geq 0} \subset X$, $\bigcap U_\ell = \{x\}$, such that for any $\ell \geq 0$ the restriction $\widehat{g}|U_\ell$ is not the identity map.
\item For any $\ell \geq 0$ there is a non-empty open subset $V_\ell \subset U_\ell$ such that the restriction $\widehat{g}|V_\ell$ is the identity map.
\end{enumerate}
\end{defn}

Since $G$ is identified with a dense subgroup $\Phi(G)$ of $E(G)$, if $ \widehat{g} \in \Phi(G)$ is non-Hausdorff, then we say that any $g\in G$ such that $\widehat{g} = \Phi(g)$ is a non-Hausdorff element in $G$.

We now list a selection of results from \cite{HL2019a,HL2019b} illustrating the classification in Definition \ref{defn-classif}.

\begin{ex}\label{thm-firstexamples}
\begin{enumerate}
\item {\it Stable actions \cite[Theorems 1.10 and 1.12]{DHL2017}:} Given any finite or any separable profinite group $H$ there exists a torsion-free finite index subgroup $\Gamma \subset \operatorname{SL}(n ,\Z)$, for $n \geq 3$, and an action $\Phi: \Gamma \to \operatorname{Homeo}(X)$ on a Cantor set $X$, such that the stabilizer group chain $K(\Phi)=\{K_\ell\}_{\ell \geq 0}$ of the action is bounded with $K_\ell = H$ for $\ell \geq 0$.
\item {\it Wild actions of finite and flat type \cite[Theorem 1.10]{HL2019a}:} There exists a torsion-free finite index subgroup $\Gamma \subset \operatorname{SL}(n ,\Z)$, for $n \geq 3$, and an uncountable collection of pairwise non-conjugate actions $\Phi_\nu: \Gamma \to \operatorname{Homeo}(X)$ of $\Gamma$ on a Cantor set $X$, such that every action $(X,\Gamma,\Phi_\nu)$ is wild of finite and flat type, that is, for any such action $(X,\Gamma,\Phi_\nu)$ the associated group chain $K(\Phi_\nu) = \{K_\ell\}_{\ell \geq 0}$ is unbounded, and for $\ell \geq 0$ the subgroup $K_\ell$ is finite with $K_\ell = Z_\ell$.
\item {\it Dynamically wild actions not of finite type \cite[Theorem 1.7]{HL2019b}:} Let $(X,G,\Phi)$ be a minimal equicontinuous action, and suppose $E(G)$ contains a non-Hausdorff element. Then the action $(X,G,\Phi)$ is dynamically wild not of finite type, that is, the associated group chain $K(\Phi) = \{K_\ell\}_{\ell \geq 0}$ is unbounded, for $\ell \geq 0$ the subgroup $K_\ell$ is infinite, and $Z_\ell$ is a proper subgroup of $K_\ell$.
\item {\it Dynamically wild actions not of finite type from iterated monodromy groups, \cite[Theorem 1.10]{Lukina2019b} and \cite[Theorem 1.7]{HL2019b}:} In $(X,G,\Phi)$, suppose $X$ is the boundary of a binary tree, and $G$ is the iterated monodromy group associated to the quadratic post-critically finite polynomial $f(x)$ with strictly pre-periodic critical orbit of cardinality at least $3$. Then $G$ contains non-Hausdorff elements, and the action $(X,G,\Phi)$ is dynamically wild not of finite type.
\end{enumerate}
\end{ex}

In the examples of dynamically wild actions in Example \ref{thm-firstexamples}, there are no details about the centralizer group of the action, in particular, it is not known if these actions are algebraically stable or algebraically wild. Our goal in this work is to construct further examples that illustrate the classification in Definition \ref{defn-classif}, paying particular attention to the properties of the centralizer group. We now state our main theorems.

We say that a countable group $H$ is $k$-generated, if it can be generated by $k$ elements. We denote by $\Alt(k)$ the alternating group on $k$ symbols.

\begin{thm}\label{thm-main1}
There exists a $2$-generated countable dense subgroup $H$ of the infinite product $\prod_{i \geq 1} \Alt(m_i)$, with $m_i \geq 7$ and $m_i \to_{i \to \infty} \infty$, and an action $(X,H,\Phi)$ of this group on a Cantor set $X$, which is wild of finite type and algebraically stable with trivial centralizer group $\U_c^x(\Phi)$, $x \in X$. 
\end{thm}

A wild action which is algebraically stable is necessarily dynamically wild, since in this case the group chain $\{Z_\ell\}_{\ell \geq 0}$ has a maximal element, while the group chain $\{K_\ell\}_{\ell \geq 0}$ is unbounded. We remark that the action in Theorem \ref{thm-main1} is dynamically wild. Comparing this result with item (2) in Example \ref{thm-firstexamples}, we observe that a wild action of finite type may be of flat type, and so algebraically wild, as in item (2) of Example \ref{thm-firstexamples}, as well as dynamically wild and algebraically stable, as in Theorem \ref{thm-main1}.

Many minimal equicontinuous group actions arise as actions on the boundary $\widehat{X}$ of a spherically homogeneous rooted tree $\tree$ with spherical index $m = (m_1,m_2,\ldots)$. The set of vertices $V(\tree)$ of such a tree is divided into \emph{levels} $V_i$, $i \geq 0$, where $V_0 = \{*\}$ is a singleton, and for $i \geq 1$ the entry $m_i$ of the spherical index is equal to the number of vertices in $V_i$ connected by edges to the same vertex in $V_{i-1}$.  The boundary $\widehat{X}$ is the set of infinite connected paths of edges in the tree $\tree$ with metric topology, see Section \ref{sec-trees} for details. With this topology, $\widehat X$ is a Cantor set. The group $\Aut(\tree)$ of automorphisms of $\tree$ consists of maps which preserve the structure of the tree $\tree$, that is, they fix the root $\{*\}$ and map infinite connected paths to infinite connected paths. Thus every element in $\Aut(\tree)$ induces a homeomorphism of $\widehat X$. It is well-known that for a spherically homogeneous tree $\tree$ with spherical index $m$ the automorphism group $\Aut(T) $ is isomorphic to the infinite wreath product of the symmetric groups $\Sym(m_i)$ on $m_i$ symbols, for $i \geq 1$, see Section \ref{subsec-wreathproduct} for details. For $i \geq 1$ let $A_i \subset \Sym(m_i)$ be a subgroup, and denote by $A^\infty$ the wreath product of $A_i$, $i \geq 1$, defined as in Section \ref{subsec-wreathproduct}.
By Remark \ref{rem-countablegen}, $A^\infty$ contains a countably generated dense subgroup, which we denote by $G$.

\begin{thm}\label{thm-main2}
Let $\tree$ be a spherically homogeneous tree with spherical index $m = (m_1,m_2,\ldots)$.
Let $G$ be a countably generated group acting minimally on the boundary $\widehat X$ of $\tree$, and suppose that the Ellis group $E(G)$ of the action is isomorphic to the wreath product $A^\infty$ of finite subgroups $A_i \subset \Sym(m_i)$, such that $A_i$ acts transitively on the set of $m_i$ symbols, for $i \geq 1$. Then the action $(\widehat X, G, \Phi)$ of $G$ on $\widehat{X}$ is dynamically wild and algebraically stable with trivial centralizer group $\U_c^x(\Phi)$, for $x \in \widehat X$. 
\end{thm}

To rephrase the statement of Theorem \ref{thm-main2}, we show that if a countable group $G$, acting effectively on the boundary of a spherically homogeneous tree, is a dense subgroup of the wreath product of finite groups, then the action is always algebraically stable with trivial centralizer group.

\begin{ex}\label{ex-finitelygen}
In Theorem \ref{thm-main2}, one can specify the conditions on the groups $A_i$ under which $G$ is finitely generated. Let $A_i = \Alt(m_i)$ for $i \geq 5$, where $m=(m_1,m_2,\ldots)$ is the spherical index. If $m_i \geq 5$ for $i \geq 1$, then each $\Alt(m_i)$ is simple, which implies that it is equal to its commutator subgroup. Then by \cite{Bondarenko2010} the infinite wreath product of the groups $\Alt(m_i)$, $ i\geq 1$, defined as in Section \ref{subsec-wreathproduct}, is topologically finitely generated. Since $\Alt(m_i) \subset \Sym(m_i)$, and $\Alt(m_i)$ acts transitively on the set of $m_i$ elements, a dense subgroup $G$ of the wreath product acts minimally on the boundary $\widehat{X}$ of the tree $\tree$. By \cite[Theorem 1.10]{Lukina2019b} any infinite wreath product of finite groups contains non-Hausdorff elements, which implies that the action of $G$ on $\widehat X$ is dynamically wild not of finite type. Theorem \ref{thm-main2} shows that this action is always algebraically stable with trivial centralizer group.
\end{ex}

Another interesting class of examples which belongs to the family of examples in Theorem \ref{thm-main2} is the following.

\begin{ex}\label{ex-arith}
For $d \geq 2$, denote by $C(d)$ the cyclic group of order $d$, and by $C(d)^\infty$ the infinite wreath product of such cyclic groups. Since $C(d)^\infty$ is a profinite group, it is topologically countably generated. Let $G$ denote a countable dense subgroup of $C(d)^\infty$, and note that $G$ acts minimally on the boundary $\widehat{X}$ of a spherically homogeneous tree $\tree$ with constant spherical index $m = (d,d,\ldots)$.
We note that actions of the wreath product $C(d)^\infty$ naturally arise in number theory and arithmetic dynamics. For instance, it was shown in \cite[Theorem 1.1]{BHL2017} that if $\varphi_d$ is a specific choice of a unicritical polynomial of prime degree $d$, and $K$ is a number field containing the $d$-th primitive root of unity, then the image of the arboreal representation of the absolute Galois group of $K$ into the group of automorphisms of a tree $\tree$ with constant spherical index $m = (d,d,\ldots)$, defined by $\varphi_d$, is a finite index subgroup of $C(d)^\infty$. If $d=3,5$ or $7$, then the image of the representation defined by $\varphi_d$ equals $C(d)^\infty$. For a family of specific quadratic polynomials $\psi_p$, whose coefficients are determined by an odd prime $p$, \cite[Theorem 1.2]{BHL2017} shows that the corresponding arboreal representation of the absolute Galois group of $K$ is equal to $C(2)^\infty$ for $p< 5000$, and also for other values of $p$ which satisfy certain congruences.
\end{ex}

Finally we construct a family of actions which are dynamically and algebraically wild with non-trivial centralizer direct limit group.

\begin{thm}\label{thm-main3}
Let $X$ and $Y$ be Cantor sets, and let $H$ and $G$ be countable groups. Suppose $(X,H,\Phi)$ is a minimal equicontinuous action which is wild of flat type, and suppose $(Y,G,\Psi)$ is a minimal equicontinuous action which is dynamically wild and algebraically stable with trivial centralizer group. Then the product action $(X \times Y, H \times G, \Phi \times \Psi)$ is dynamically wild and algebraically wild. That is, the stabilizer and centralizer subgroups $\U_s^{(x,y)}(\Phi \times \Psi)$ and $\U_c^{(x,y)}(\Phi \times \Psi)$ are both unbounded, and $\U_s^{(x,y)}(\Phi \times \Psi) \ne \U_c^{(x,y)}(\Phi \times \Psi)$, for $(x,y) \in X \times Y$. In addition, if $(X,H,\Phi)$ is wild of finite type, then $(X\times Y,H\times G, \Phi \times \Psi)$ is dynamically wild and algebraically wild of algebraic finite type.
\end{thm}

For example, in Theorem~\ref{thm-main3}, we can take $(X,H,\Phi)$ given by item~(2) of Example~\ref{thm-firstexamples}, and $(Y,G,\Psi)$ given by Theorem~\ref{thm-main2}.

We finish by discussing open problems. 

As explained in detail in Section \ref{sec-trees}, every equicontinuous minimal Cantor action can be represented as an action on the boundary of a spherically homogeneous tree $\tree$ with spherical index $m = (m_1,m_2,\ldots)$. Examples in Theorems \ref{thm-main1} and \ref{thm-main2} are constructed as such actions. As a rule, a representation of an action on the boundary of a tree is not unique; in particular, the spherical index of the tree may be different for different representations of the same action. For instance, if an action admits a representation on the boundary of a tree with constant spherical index $m = (d,d,\ldots)$, then it admits a representation on the boundary of a tree with strictly increasing spherical index $m' = (d, d^2, d^3, \ldots)$. However, if  the entries of the spherical index $m$ are distinct primes, then an action does not admit a representation on the boundary of a tree with constant spherical index. For example, the wild actions of flat and finite type in item (2) of Example \ref{thm-firstexamples} have this property, that is, they have representations onto the boundary of a tree where the spherical index is an increasing sequence of distinct primes, and so they do not admit a representation on the boundary of a tree with constant spherical index. We show in Proposition~\ref{p:limsupmn} that the family of examples in Theorem \ref{thm-main1} also do not admit representations on the boundary of a tree with constant spherical index. On the other hand, the examples in Theorems \ref{thm-main2} include actions of wreath product groups on trees with constant spherical index.

Actions which admit representations on the boundary of a tree with constant spherical index are probably most interesting from the point of view of applications, as they may arise as dynamical objects, associated to iterations of the same map (attractors in dynamical systems, arboreal representations of Galois groups associated to a polynomial, etc.). Therefore, the following problem is of interest.

\begin{prob}\label{prob-constant}
In the classification of Definition \ref{defn-classif}, determine what classes of actions do not admit representations on the boundary of a tree with constant spherical index.
\end{prob}

For instance, so far we do not have examples of wild actions of flat type or of finite type which admit representations onto the boundary of a tree with constant spherical index, but we expect that such examples exist.

A related question is whether it is possible to realize a given Cantor group action as a holonomy action on a transverse section of a minimal set of a foliation of a smooth manifold. In particular, we can ask this question for the families of examples constructed in Theorems \ref{thm-main1}, \ref{thm-main2} and \ref{thm-main3}. A necessary condition for this question having a positive answer is that the Cantor set with the metric, compatible with the action, admits a bi-Lipschitz embedding into $\mathbb{R}^n$ \cite{Hurder2017}. In our examples metrics on Cantor sets are determined by tree structures, and it is explained in Remark \ref{rem-assouad} that if the boundary of a spherically homogeneous tree with metric defined in Section \ref{sec-chains} admits a bi-Lipschitz embedding into $\mathbb{R}^n$, then the tree must have bounded spherical index. For the family of examples in Theorem \ref{thm-main1}, by Proposition \ref{p:limsupmn} any representation of the action onto the boundary of a spherically homogeneous tree has unbounded spherical index. Thus no action in Theorem \ref{thm-main1} can be realized as a holonomy action on a transverse section of a foliation of a smooth manifold.

In Example \ref{ex-arith}, a dense subgroup $G$ of the wreath product $C(d)^\infty$ acts on the boundary of a spherically homogeneous tree with constant spherical index $m=(d,d,\ldots)$, but the acting group $G$ is infinitely generated. Therefore, the family in Example \ref{ex-arith} cannot be realized as holonomy actions on a transverse section of a foliation of a compact smooth manifold, since such actions must be compactly generated.

Among actions described in Theorem \ref{thm-main2} there are actions of finitely generated groups, for instance, a dense subgroup $G$ of the infinite wreath product of alternating groups in Example \ref{ex-finitelygen}. In this example, we can also choose the spherical index of the tree to be constant. The action of a dense subgroup of an infinite wreath product group is always wild, so by the observation before \cite[Corollary 1.12]{HL2019a} it cannot be realized as a holonomy action on a transverse section of a real-analytic foliation of a compact manifold. The analogous question for smooth foliations of compact manifolds is open.

\begin{prob}\label{prob-embedding}
Let $\tree$ be a spherically homogeneous tree with constant spherical index $m = (d,d,\ldots)$, for $d \geq 2$, and let a finitely generated group $G \subset \Aut(\tree)$ act on the boundary $\widehat X$ of $\tree$ minimally. Suppose the closure $E(G)$ of the action is isomorphic to the infinite wreath product $A^\infty$, where $A \subset \Sym(d)$ is a finite permutation group. Determine if the action of $G$ on $\widehat X$ is conjugate to a holonomy action on a subset of a transverse section of a $C^r$ foliation of a manifold, for $r = 1,2,\ldots, \infty$.
\end{prob}

\emph{Acknowledgements}. The first and the third authors thank for hospitality the Ritsumeikan University, where a part of this work was done. We are grateful to the anonymous referee for careful reading and helpful suggestions, which improved the clarity of the paper.

\section{Group chains}\label{sec-chains}

Below we describe a basic class of examples of equicontinuous group actions on Cantor sets. We give a brief overview of the construction without proofs. Proofs and further details can be found, for instance, in \cite[Section 3]{HL2019b}. In this section we denote by $\X$ any Cantor set, equipped with an equicontinuous action of a countable group $G$, and we reserve the notation $X$ for the Cantor set obtained by the inverse limit construction in \eqref{eq-invlimx}.

\subsection{Actions from group chains}
Let $G$ be a countably generated group. For $n \geq 1$, let $G_n$ denote a finite index subgroup of $G$, such that these subgroups form a decreasing chain $
\cG: G=G_0\supset G_1\supset G_2 \supset \ldots\;
$ 
We assume that for any $n \geq 0$ the group $G_{n+1}$ is a proper subgroup of $G_n$, that is, $|G_n : G_{n+1}| \geq 2$.

For $n\geq 0$ the coset space $G/G_n$ is a finite set. If, in addition, $G_n$ is a normal subgroup of $G$, then $G/G_n$ is a finite group. We are interested predominantly in the case when $\cG$ is \emph{not} a chain of normal subgroups.  

For each $n \geq 0$, there are surjective maps of coset spaces $\nu^{n+1}_n:G/G_{n+1} \to G/G_n$, given by the coset inclusion $g G_{n+1} \to g G_n$. By the standard argument the inverse limit space
  \begin{align}\label{eq-invlimx}X  & = \lim_{\longleftarrow}\{\nu^{n+1}_n: G/G_{n+1} \to G/G_n  \mid n\geq 0\} \\  \nonumber  & \phantom{--------} = \left\{(g_0 G_0 ,g_1 G_1,\ldots) \in \prod_{n \geq 0} G/G_n \mid \nu^{n+1}_n (g_{n+1}G_{n+1}) = g_nG_n \right\}\end{align}
is a Cantor set with respect to the product (Tychonoff) topology. 
The group $G$ acts by left translations on each coset space $G/G_n$, and so there is an induced group action
  \begin{align}\label{eq-Gaction}G \times X \to X: (h, (g_n G_n)_{n \geq 0}) \mapsto (hg_n G_n)_{n\geq 0}, \end{align}
where $(g_n G_n)_{n \geq 0} = (g_0 G_0,g_1G_1,\ldots)$ denotes a sequence which is an element of $X$. Since the action of $G$ on every coset space $G/G_n$, $n \geq 0$, is transitive, the action of $G$ on $X$ is minimal. The space $X$ can be given an ultrametric $D$, for example, by
 \begin{align}\label{eq-metricd} D((g_n G_n)_{n \geq 0}, (h_n G_n)_{n \geq 0}) = \frac{1}{2^m}, & & \textrm{ where }m = \min\{n \geq 0 \mid g_n G_n \ne h_n G_n\}, \end{align}
that is, this metric measures for how long two sequences of cosets $(g_nG_n)_{n \geq 0}$ and $(h_n G_n)_{n \geq 0}$ coincide. Since $G$ acts on each coset space $G/G_n$ by permutations, it acts on $X$ by isometries relative to the metric $D$. It follows that the action of $G$ is equicontinuous. We denote this action by $(X,G)$, omitting the notation for the homomorphism $G \to \operatorname{Homeo}(X)$, as it is determined by the construction. 

Basic sets of the product topology on $X$ are given by
\begin{align}\label{eq-basicset}U_{g,m} = \{(g_n G_n)_{n \geq 0} \in X \mid g_m G_m = g G_m\}, \end{align}
 where $m \geq 0$ and $g \in G$. The isotropy subgroup of the action of $G$ at a basic set $U_{g,m}$ is given by
  \begin{align}\label{eq-isotropy}\Iso(U_{g,m}) = \{h \in G \mid h \cdot (g_n G_n)_{n \geq 0} \in U_{g,m} \textrm{ for all }(g_n G_n)_{n \geq 0} \in U_{g,m}\}.\end{align}
Then $\Iso(U_{e,m}) = G_m$, where $e$ is the identity in $G$, and $\Iso(U_{g,m}) = g G_mg^{-1}$ for any $g \in G$.

If $(\X, G , \Phi)$ is any minimal equicontinuous action of a countably generated group $G$ on a Cantor set $\X$, then one can associate to the action a group chain as above. The procedure is folklore, and it is described in detail for instance in \cite[Appendix]{DHL2016}. The idea of the construction is as follows: given an open neighborhood $V \subset \X$ of $x$, there exists a clopen subset $x \in U \subset V$ such that the set of return times of the action of $G$ to $U$ is equal to the subgroup $G_U$. This subgroup is precisely the isotropy group of the action at $U$ defined similarly to the one in formula \eqref{eq-isotropy}. It follows that there exists a decreasing chain $U_0 = \X \supset U_1 \supset U_2 \supset \ldots$ of clopen neighborhoods of $x$ in $X$, with $\bigcap_{n \geq 0} U_n = \{x\}$, such that $G_n : = G_{U_n} = \Iso(U_n)$ is the isotropy subgroup of the action of $G$ at $U_n$. For each $U_n$, $n \geq 0$, the translates $\{g \cdot U_n\}_{g \in G}$ form a finite partition of $\X$, and it follows that the index of $G_n$ in $G$ is finite. Thus we obtain a decreasing sequence of finite index subgroups $G_0 = G \supset G_1 \supset \ldots$, associated to the action, and a dynamical system $(X,G)$ defined by the formulas \eqref{eq-invlimx} and \eqref{eq-Gaction}. The last step of the construction is to build a homeomorphism $\phi: \X \to X$ such that the actions of $G$ on $\X$ and $X$ commute and such that $\phi(x) = (eG_n)_{n \geq 0} \in X$, where $eG_n$ denotes the coset of the identity $e$ in $G/G_n$, see  \cite[Appendix]{DHL2016} for details.

\subsection{Groups $E(G)$ and $E(G)_x$ from group chains}\label{subsec-EG}
Given an action $(\X,G,\Phi)$ one can use the associated group chain $\{G_n\}_{n \geq 0}$ to compute the Ellis group $E(G)$ and the isotropy group $E(G)_x$ of the action of $E(G)$ at $x \in \X$ as follows. 

For each $n \geq 0$, if $h \in g\,G_n$, then $g\, G_n \, g^{-1} = h\, G_n\, h^{-1}$, which implies that the number of conjugacy classes of $G_n$ in $G$ is less or equal to the index of $G_n$ in $G$. Therefore, $G_n$ has a finite number of distinct conjugates in $G$. Each conjugate of $G_n$ has the same index in $G$ as $G_n$. Then the intersection of a finite number of finite index subgroups
\[
C_n=\bigcap_{g\in G} gG_ng^{-1}
\] 
is a finite index normal subgroup in $G$. It is immediate that $C_n$ is the largest subgroup of $G_n$ which is normal in $G$.

Since $C_n$ is a normal subgroup of $G$, the coset space $Q_n=G/C_n$ is a finite group. Let $\pi^{n+1}_n\colon Q_{n+1}\to Q_n$ be the homomorphism of finite groups induced by the inclusion $C_{n+1}\subset C_n$. Then the inverse limit space
\begin{align*}
\widehat G_\infty &= \varprojlim\{\pi_{n+1}\colon Q_{n+1}\to Q_n\mid n\geq 0\} \\ & \phantom{--------} = \left\{(g_0 C_0 ,g_1 C_1,\ldots) \in \prod_{n \geq 0} G/C_n \mid \pi^{n+1}_n (g_{n+1}C_{n+1}) = g_nC_n \right\}
\end{align*}

is a profinite group. Also, consider the inverse limit group
\[
	\mathcal{D}_x=\varprojlim\{\pi^{n+1}_n\colon G_{n+1}/C_{n+1}\to G_n/C_n\mid n\geq 0\}\subset \widehat G_\infty\;
	\]
called the \emph{discriminant group} of the action. The following result was proved in \cite{DHL2016}.

\begin{thm}[{\cite[Theorem~4.4]{DHL2016}}]\label{thm-quotientspace}
Let    $(\X,G,\Phi)$ be an   equicontinuous minimal Cantor action, let $x \in \X$,  and let $\cG \equiv \{G_{n}\}_{n \geq 0}$ be an associated group chain at $x$.
Then there is an isomorphism of topological groups $\psi: E(G) \to \widehat G_{\infty}$ such that the restriction $\psi: E(G)_x \to \cD_x$ is an isomorphism.
\end{thm}

We construct the actions in Theorems \ref{thm-main1}, \ref{thm-main2} and \ref{thm-main3} using the method of group chains, described above.

\section{Trees}\label{sec-trees}

In this section we show that every equicontinuous minimal action can be represented as an action on the boundary of a rooted spherically homogeneous tree $\tree$. 

\subsection{Spherically homogeneous trees}
A \emph{spherical index} is a sequence ${m} = (m_1,m_2,\ldots)$ of natural numbers, where $m_i \geq 2$ for $i \geq 1$. A spherical index $m$ defines a \emph{spherically homogeneous rooted tree} $\tree$ as follows. Let $V_0 = X_0$ be a singleton, and for $n \geq 1$ fix an identification $V_n \cong \prod_{i=0}^n X_i$, where $X_i$ is a finite set with $m_i$ elements. Denote by $\pr_n: \prod_{i=0}^{n+1} X_i \to \prod_{i=0}^{n} X_i$ the projection, and join $v_{n+1} \in V_{n+1}$ and $v_n \in V_n$ by an edge if and only if $\pr_n(v_{n+1}) = v_n$. Thus every vertex in $V_{n}$ is connected by edges to precisely $m_{n+1}$ vertices in $V_{n+1}$, and every vertex in $V_{n+1}$ is connected by an edge to precisely one vertex in $V_n$. 

A path in $\tree$ is an infinite sequence of vertices $(v_n)_{n \geq 0} = (v_0,v_1,v_2,\cdots)$ such that $v_{n}$ and $v_{n+1}$ are connected by an edge, for $n \geq 0$. We always assume that $v_n \in V_n$ for $n \geq 0$, so our paths do not backtrack and have no self-intersections. Denote by $\widehat{X}$ the set of all such sequences. Using the identifications $V_n \cong \prod_{i=0}^n X_i$, one can show that $\widehat X \cong \prod_{i \geq 0} X_i$, and so it is a Cantor set in the product (Tychonoff) topology. For a vertex $v\in \tree$, let $\tree_v\subset \widehat X$ be the set of infinite paths that pass through $v$. Then the family $\{\tree_v, v\in\tree\}$ is a basis of clopen subsets of $\widehat X$. The Cantor space $\widehat{X}$ is called the \emph{boundary} of the tree $\tree$. Using a slightly different formalism, the boundary of $\tree$ can also be seen as the inverse (projective) limit space
\begin{align}\label{eq-treeinvlim} \widehat X \cong \lim_{\longleftarrow}\{{\rm pr}_n: V_{n+1} \to V_n \mid n\geq 0 \} = \{(v_n)_{n\geq 0} \in \prod_{n \geq 0} V_n \mid {\rm pr}_n(v_{n+1}) = v_n\}, \end{align}
which is sometimes useful in arguments, see \cite[Section 1]{RZ} for the basics about the inverse limits.

A root-preserving automorphism of $\tree$ is a map $a: \tree \to \tree$ which maps vertices to vertices and edges to edges in such a way that the root $v_0$ is fixed and infinite paths in $\tree$ are mapped to infinite paths. The restriction of such a map to each vertex level $V_n$ is a permutation, and it follows that $a$ induces a homeomorphism $\phi_a: \widehat{X} \to \widehat{X}$ of the boundary. The group of all root-preserving automorphisms $\Aut(\tree)$ can be expressed as the infinite wreath product of the symmetric groups, see Section \ref{subsec-wreathproduct} for details.

\subsection{Equicontinuous minimal actions on the boundary of a tree} Given a group chain $\{G_n\}_{n \geq 0}$ as in Section \ref{sec-chains}, we can associate to it a spherically homogeneous tree and equip the tree with an action by root-preserving automorphisms as follows.

Let $X_0$ be a singleton, and for $n \geq 1$, let $X_n$ be a set of cardinality $|G_{n-1}/G_n|$. Then $V_n = \prod_{i=0}^n X_i$ has the same cardinality as $G/G_n$, for $n \geq 0$. For $n \geq 0$, fix bijections $\lambda_n: V_n \to G/G_n$ in such a way that they respect the identifications $V_n = \prod_{i=0}^n X_i$, that is, for any $w, v \in V_{n+1}$ we have $\pr_{n}(v) = \pr_{n}(w)$ if and only if the cosets $\lambda_{n+1}(v)$ and $\lambda_{n+1}(w)$ in $G/G_{n+1}$ are in the same coset of $G/G_{n}$. Then the inverse limit map ${\displaystyle \lambda_\infty = \lim_{\longleftarrow} \lambda_n: \widehat X \to X}$ is a homeomorphism of Cantor sets, where $X$ is the inverse limit space \eqref{eq-invlimx}.  Since $G$ acts on each $G/G_n$ by permuting cosets, the pullback of this action to $V_n$ permutes the vertices in $V_n$ in such a way that for each $g \in G$ the vertices $v,w \in V_{n+1}$ are connected to the same vertex $z \in V_{n}$ if and only if $g \cdot v$ and $g \cdot w$ are connected to the same vertex in $V_{n}$, for $n \geq 0$. It follows that the induced action of $G$ on $\tree$ is by root-preserving automorphisms. 

This construction together with Section \ref{sec-chains} shows that every minimal equicontinuous action $(\X,G,\Phi)$ on a Cantor set $\X$ admits a representation as an action by root-preserving automorphisms on the boundary $\widehat X$ of an appropriately chosen spherically homogeneous tree $\tree$.

\section{Wild actions from product groups}\label{ss:productactions}

Our strategy in this section is to define an action of a profinite group first, and then to find a countable dense subgroup which generates the profinite group. This subgroup with discrete topology is the acting countable group. The profinite group in this section is a product of finite groups.

\subsection{Product groups}\label{subsec-product}
For $i \geq 1$, let $X_i$ be a finite set with $|X_i| \geq 2$. Let $F_i \subset \Sym(|X_i|)$ be a finite group, then the induced action of $F_i$ on $X_i$ is effective.  The infinite product $\widehat{G} = \prod_{i \geq 1} F_i$ is a profinite group, its elements are sequences $(g_i)_{i \geq 1} = (g_1,g_2,\ldots)$. The product space $\widehat{X} = \prod_{i\geq 1} X_i$ with Tychonoff topology is a Cantor set, where elements are sequences $\widehat{x} = (x_i)_{i\geq 1} = (x_1,x_2,\ldots)$. The actions of distinct factors in $\prod_{i \geq 1} F_i$ on factors of $\prod_{i \geq 1} X_i$ commute, and we define the action of $\widehat{g} \in \widehat{G}$ on the space of sequences $\widehat{X}$ by 
\begin{equation}\label{productaction}
\widehat g \cdot \widehat x=(g_1,g_2,\ldots) \cdot (x_1, x_2,  \ldots )=(g_1 \cdot x_1, g_2 \cdot x_2, \ldots).
\end{equation}
Then the restriction of the action of $\widehat{G}$ to the finite product $\prod_{i=1}^n X_i$ is the action of the finite group $\prod_{i=1}^n F_i$. 


Recall that an action of a group $G$ on a set $Y$ is \emph{transitive} if for any $y_1,y_2 \in Y$ there is $g \in G$ such that $g \cdot y_1 = y_2$. It is straightforward that if $G_1$ and $G_2$ are groups acting transitively by homeomorphisms on the sets $Y_1$ and $Y_2$ respectively, then the product action of $G_1 \times G_2$ on $Y_1 \times Y_2$ is transitive.

We will use the following result from \cite{Miller1928}. As in the Introduction, we denote by $\Alt (n)$ the alternating group on $n$ symbols. For $\sigma \in \Alt(n)$, we denote by $\ell(\sigma)$ the order of $\sigma$.

\begin{thm}\label{t:primes}
	Let $\ell_1, \ell_2 \geq 3$ be odd integers, and let $n\in \mathbb{N}$ satisfy $\ell_1,\ell_2\leq n$ and $n< \ell_1+\ell_2$. Then there are cyclic permutations $\sigma_1, \sigma_2\in \Alt (n)$ such that $\ell(\sigma_1) = \ell_1$ and $ \ell(\sigma_2) = \ell_2$, and $\sigma_1$ and $\sigma_2$ generate $\Alt(n)$, that is, $\Alt(n)=\langle \sigma_1,\sigma_2\rangle$.
\end{thm}

\subsection{Dynamically wild action of finite type} In this section we prove Theorem \ref{thm-main1} by constructing a finitely generated  dense subgroup $H$ of the infinite product $\prod_{n \geq 1} \Alt(d_n)$, with $d_n \geq 5$ and $d_n \to_{n \to \infty} \infty$, and an action of $H$ on a Cantor set $X$ which is wild of finite type and algebraically stable with trivial centralizer direct limit group. 


Choose an injective map $p\colon \{1,2\}\times\N\to \N$ taking values on  odd primes, and let $o\colon\N\to\N$ be a map satisfying
\begin{equation}\label{po}
p(1,n),p(2,n)\leq o(n)\quad \text{and}\quad o(n)<p(1,n)+p(2,n)
\end{equation} 
for all $n\in\N$. For example, one can take $o(n) = p(1,n)+p(2,n)-1$, for $n \geq 1$.

Let $X_n$ be a finite set of cardinality $o(n)$, and let $A_n := \Alt(o(n))$ be the alternating group on $o(n)$ symbols. Then $A_n$ acts transitively on $X_n$ by even permutations. 

By Theorem~\ref{t:primes}, there exist two sequences $\{\sigma_{1,n}\}_{n \geq 1}, \{ \sigma_{2,n}\}_{n \geq 1}$, $n\in \N$, of permutations of $X_n$ such that
\[
 \langle \sigma_{1,n}, \sigma_{2,n}\rangle = A_n \quad \text{and} \quad \ell(\sigma_{i,n})=p(i,n), \quad \text{for} \quad i=1,2.
\]
For $i=1,2$, define $\sigma_i=(\sigma_{i,1},\sigma_{i,2},\ldots)\in \prod_{n \geq 1} A_n$ be two elements of the infinite product group, and let $H = \langle \sigma_1,\sigma_2 \rangle $ denote the subgroup of the product group that they generate. Then $H$ acts on the Cantor set $\widehat{X}= \prod_{n \geq 1} X_n$ by~\eqref{productaction}. 

Let $X_0$ be a singleton. We identify each finite product $\prod_{i=1}^n X_i$ with the vertex set $V_n = \prod_{i \geq 0}^n X_i$ of a spherically homogeneous tree $\tree$ with spherical index $m = (o(1), o(2),\ldots)$, see Section \ref{sec-trees} for details about trees. Under this identification, the product space $\widehat{X} = \prod_{i\geq 1} X_i$ is identified with the boundary $\prod_{i\geq 0} X_i$ of $\tree$. We denote the boundary also by $\widehat X$. Let $A_0 = \{e_0\}$ be the trivial group, then the action of $\prod_{i \geq 0} A_i$ on the boundary $\prod_{i\geq 0} X_i$ is conjugate to the action of $\prod_{i \geq 1} A_i$ on $\widehat X = \prod_{i\geq 1} X_i$, and there is an induced action on $\tree$ of a subgroup isomorphic to $H$, which we denote also by $H$. This subgroup is generated by the elements $(e_0,\sigma_{i,1},\sigma_{i,2},\ldots)$ for $i=1,2$. Again, we use the same notation $\sigma_i$ for $(\sigma_{i,1},\sigma_{i,2},\ldots)$ and $(e_0,\sigma_{i,1},\sigma_{i,2},\ldots)$ acting on $\prod_{i \geq 1} X_i$ and $\prod_{i \geq 0} X_i$ respectively, for $i =1,2$.

The group of all permutations of $V_n$ is isomorphic to the symmetric group $\Sym(|V_n|)$. For each $\widehat{h} \in H$, denote by $\sigma_{\widehat{h},n}$ the permutation of $V_n$ induced by the action of $\widehat{h}$. Let 
  $$H^{[n]} = \{\sigma_{\widehat{h},n} \in \Sym(|V_n|) \mid \widehat{h} \in H\}.$$ 
For $n \geq 1$, consider the restricted action $(V_n,H^{[n]})$ of $H$ on the finite set $V_n$.

\begin{lem}\label{lem-finiteproduct}
	For $n\in\N$, we have $H^{[n]} = A_0 \times A_1\times \cdots \times A_n \cong  A_1\times \cdots \times A_n$.
\end{lem}
\begin{proof}
By construction, $H^{[n]} \subseteq A_0 \times A_1 \times \cdots \times A_n$, so we have to show the converse implication.

Since $\sigma_{1,i}$ and $\sigma_{2,i}$ generate $A_i$ for $1\leq i\leq n$, and the actions of $A_i$ on the corresponding factors of $\prod_{i=0}^n X_i$ commute, every element $(e_0,g_1,\ldots,g_n)\in A_0 \times A_1 \times \cdots \times A_n$ can be expressed as a composition of elements of the form $(e_0,e_1,\ldots,\sigma_{a,k} ,\ldots, e_n)$, where $a\in\{1,2\}, 1\leq k\leq n$ and $e_i$ is the identity permutation in $\Sym(|X_i|)$, $1 \leq i \leq n$. We now show that any such element is contained in $H^{[n]}$, then the statement of the lemma follows.

Since all primes $p(i,j)$ are distinct, we can apply the Chinese Remainder Theorem in the following way: for a given $(a,k)$, $a \in \{1,2\}$ and $1 \leq k \leq n$, choose  $s(a,k)\in \N$ such that 
  $$s(a,k)\equiv1\mod p(a,k)$$ 
and for $i \ne k$, $1 \leq i \leq n$,
  $$s(a,k)\equiv0\mod p(a,i).$$ 
Then, computing the $s(a,k)$-th power of $\sigma_a$ we obtain
  $$\sigma_a^{s(a,k)}=(e_0,e_1,\ldots, \sigma_{a,k},\ldots, e_n, g_{n+1},\ldots),$$ 
which shows that $(e_0,e_1,\ldots, \sigma_{a,k},\ldots, e_n) \in H^{[n]}$.
\end{proof}

\begin{lem}\label{lem-min}
	The action $(\widehat{X},H)$ on the boundary $\widehat X$ of the tree $\tree$ is minimal.
\end{lem}
\begin{proof}
	It is enough to prove that $H^{[n]}$ acts transitively on each finite vertex set $V_n = \prod_{i=0}^n X_{i}$. By Lemma \ref{lem-finiteproduct} we have $H^{[n]} = A_0 \times A_1\times \cdots \times A_n$. Since $A_i$ acts transitively on each $X_i$, the product $A_0 \times A_1 \times \cdots \times A_{n}$ acts transitively on $V_n$.
\end{proof}

\begin{lem}\label{lem-closure}
The group $H$ is dense in $\prod_{n \geq 0} A_n$.
\end{lem}

\begin{proof}
Consider the inclusion $\iota: H \to \prod_{i \geq 0} A_i$, and the projection maps $\phi_n: \prod_{i \geq 0} A_i \to \prod_{i =0}^n A_i$. By Lemma \ref{lem-finiteproduct} the compositions $\phi_n \circ \iota$ are surjective, then by  \cite[Lemma 1.1.7]{RZ} $H$ is dense in $\prod_{i \geq 0} A_i$.
\end{proof}

We will need the following technical lemma for the isotropy group of a point $x \in X_n$ for the action of $A_n$.

\begin{lem}\label{lem-noncommute}
	Let $n\geq 1$,  let $x\in X_n$, and let $g\in A_n$ be such that $g\neq e_n$ and $g(x)=x$. Then there is some $\tau\in A_n$ such that $\tau(x)=x$ and $g\tau\neq \tau g$.
\end{lem}
\begin{proof}
	Since $g$ is not the identity transformation, let  $y\in X_n$ be such that $y\neq g(y)$. The group
	$A_n$ is the alternating group on $o(n)\geq p(1,n),p(2,n)$ symbols, where $p(1,n)$ and $p(2,n)$ are distinct odd primes, so $o(n)\geq 5$ and $X_n$ contains at least $5$ elements. 
	Let $u\in X_n\setminus \{x,y,g(y),g^{-1}(y)\}$. The cycle $\tau = (y \, g^{-1}(y) \, u)$ is an even permutation of $X_n$, which fixes $x$. Since $g \circ \tau(y) = y$ and $\tau \circ g(y) = g(y)$, the lemma is proved.
\end{proof}

Choose a point $\widehat x \in \widehat X$ in the boundary of the tree $\tree$, and consider the action $(\widehat X,H)$. 

\begin{prop}\label{prop-finite} Let $\widehat x \in \widehat X$, and let $K_n$ and $Z_n$ be the subgroups of the closure $\prod_{i \geq 0}A_i$ defined by \eqref{eq-Kchain} and \eqref{eq-Zchain} respectively, for $n \geq 0$. Then for every $n \geq 0$
	the group $K_n$ is finite with $|K_n| \to_{n \to \infty} \infty$. For every $n \geq 0$ the group $Z_n$ is trivial.
\end{prop}
\begin{proof}
Let $\widehat x = (x_n)_{n\geq 0}$, where $x_n \in X_n$. Denote by $D_n = \{ h \in A_n \mid h \cdot x_n = x_n \}$ the isotropy subgroup of the action of $A_n$ on $X_n$. Note that $D_0 = \{e_0\}$ is trivial, since $A_0$ is trivial.

The closure of the action of $H$ is $\prod_{i \geq 0}A_i$. Let $\cD_{\widehat x}$ be the isotropy group of the action of $H$ at $\widehat x$. Since the actions of the factors in $\prod_{i \geq 0}A_i$ on the factors of $\widehat{X}$ commute, we have $\cD_{\widehat x} = \prod_{i \geq 0} D_i$. We now compute the group chains representing $\U_s^{\widehat x}$ and $\U_c^{\widehat x}$.

Define a decreasing sequence of clopen neighborhoods of $\widehat{x}$ by
  \begin{align}\label{eq-nbhdzero1}U_n = \{ \widehat{w} = (w_0,w_1,w_2,\ldots) \in \widehat{X} \mid w_i=x_i \, \textrm{ for } \, 0 \leq i \leq n\}. \end{align}
Suppose $\widehat{g} = (g_i)_{i \geq 0} \in \cD_{\widehat x}$, then $(g_i)_{i =0}^n \in \prod_{i=0}^n D_i$. Because of the product structure of $\cD_{\widehat x}$, if for some $m > n$ we have $g_m \ne e_m$, then $\widehat{g}$ acts non-trivially on a clopen subset of $U_n$, and so $\widehat{g} \notin K_n$. It follows that
\[
K_n=\prod_{0\leq i\leq n}D_i\times \prod_{i>n}\{e_i\}\;
\]
is isomorphic to a finite product of finite groups, and so $K_n$ is finite for each $n \geq 1$. Since $|X_i| = o(i) \geq 5$ for $i \geq 1$, and $A_i$ contains all $3$-cycles, for $i \geq 1$ every $D_i$ contains non-trivial elements, which fix $x_i$ and act non-trivially on the complement of $x_i$ in $X_i$. It follows that $|K_n| \to \infty$ as $n \to \infty$. 

Let us show that the centralizer group $Z_n$ is trivial.  Note that the subgroup of elements in $E(G)$ which preserve $U_n$ is given by $\widehat{U}_n = \prod_{0 \leq i \leq n}D_i \times \prod_{i > n} A_i$. Suppose that there is a non-trivial $\widehat{h}\in K_n$, then 
\[
\widehat{h}=(e_0,h_1,\ldots,h_n,e_{n+1},\ldots)\in \prod_{0\leq i\leq n}D_i\times \prod_{i>n}\{e_i\}\;,
\]
where  $h_i$ is not the identity for at least one index $i$, $1\leq i\leq n$. Choose one such $i$, then by Lemma \ref{lem-noncommute} we can find $\tau_i \in D_i$ such that $\tau_i \, h_i \ne h_i \, \tau_i$.  Then $\widehat\tau:=(e_0,e_1,\ldots, e_{i-1},\tau_i,e_{i+1},\ldots)\in K_n \subset \widehat{U}_n$ satisfies 
$\widehat\tau \, \widehat h \neq \widehat h \, \widehat\tau$. This implies that the adjoint action $Ad(\widehat{h})$ is non-trivial on $\widehat U_n$, and $\widehat{h}$ is not in $Z_n$. Since this holds for any $\widehat h \in K_n$, it follows that $Z_n$ is a trivial group.
\end{proof}

\begin{proof}[Proof of Theorem \ref{thm-main1}] The action $(\widehat{X},H)$ defined above is equicontinuous since it is an action on the boundary of a rooted spherically homogeneous tree, and it is minimal by Lemma \ref{lem-min}.

By Proposition \ref{prop-finite} we have that for $n \geq 0$ the group $K_n$ is finite, and $|K_n| \to \infty$ as $n \to \infty$. Therefore, the stabilizer group chain $\U_s^{\widehat x}$ is unbounded and the action 
$(\widehat{X},H)$ is wild of finite type. On the other hand, $Z_n$ is trivial for all $n \geq 0$, so $(\widehat{X},H)$ is algebraically stable and dynamically wild.
\end{proof}

We close this section by proving that, for any choice of functions $p(1,n)$, $p(2,n)$, $o(n)$, the resulting action of $H$ on $\widehat{X}$ is not conjugate to an action on the boundary of a tree with constant spherical index. 

For this proof we will need the following notions. 

Let $S=\{s_1,\ldots,s_n\}$ be a finite set, and let $\Gamma$ be a group  acting  on $S$  by permutations. Let $\mathcal{B}=\{B_1,\ldots,B_m\}$ be a partition of $S$; that is, $S=\bigcup_{i=1}^{m} B_i$ and $B_i\cap B_j=\emptyset$ for $i\neq j$. We say that the action of $\Gamma$ \emph{preserves the partition} $\mathcal{B}$ if every element of $\Gamma$ permutes the sets in $\mathcal{B}$. More precisely, for every  $\gamma\in\Gamma$ and every $i=1,\ldots,m$, there is some $j$ such that $\gamma \cdot B_i=B_j$. We say that the action of $\Gamma$ on $S$ is \emph{primitive} if it is transitive on the elements of $S$ and the only partitions it preserves are the trivial partitions $\{S\}$ and $\{\{s_1\}, \ldots, \{s_n\}\}$. The following is a well-known result which can be found, for instance, in ~\cite[Theorems~9.6,~9.7]{Wieland}.
    
\begin{lem}\label{lem-primitive}
 For $n>2$, the action of the alternating group $\Alt(n)$ on $\{1,\ldots,n\}$ is primitive.
\end{lem}

The following statement follows from the definition of a primitive action.

\begin{lem}\label{cor-trivquotient}
Let $S$ and $Q$ be finite sets with an action of a group $H$. Let $f : S \to Q$ be an equivariant map, i.e., for any $ \tau \in H$ and any $s \in S$ we have $\tau \cdot f(s) = f(\tau \cdot s)$. If $f$ is surjective and the action of $H$ on $S$ is primitive, then either $|Q| = |S|$ or $Q$ is a singleton.
\end{lem}


We now show that if an action of the group $H$ on the boundary $\widehat{X}'$ of a tree $\tree'$ is conjugate to the action $(\widehat X, H)$ in Theorem \ref{thm-main1}, then the spherical index $m' = (m_1',m_2',\ldots)$ of $\tree'$ cannot be constant.

To establish notation, let $X'_0$ be a singleton set and, for $n\geq 1$, let $X'_n$ be a sequence of sets with cardinality $|X'_n|=m_n'$. Then the set of vertices of $\tree ' $ at level $n$ is $V'_n=\prod_{i=0}^{n}X_i'$, and the set of infinite paths $\widehat X' \cong \prod_{i \geq 0} X_i'$. 
For the action $(\widehat{X},H)$ defined in Theorem \ref{thm-main1}, suppose there is a homeomorphism $\phi: \widehat X \to \widehat X'$, such that for every $\widehat h \in H$ the composition $\phi \circ \widehat h \circ \phi^{-1}$ is an automorphism of the tree $\tree'$. Then there is an induced injective map $\phi_*: H \to \operatorname{Homeo}(\widehat X')$, and so $H$ acts on $\widehat X '$ by homeomorphisms.

In the proof, we use the property that every homeomorphism $\phi: \widehat X \to \widehat X'$ of the boundaries of trees is induced by a tree map. This is a consequence of the fact that $\widehat X$ (resp.\ $\widehat X'$) is a product of finite sets $X_i$ (resp.\ $X_i'$), $i\geq 0$, and so it can be represented as the inverse limit of finite sets $V_n = \prod_{i=0}^n X_i$ (resp.\ $V_n' = \prod_{i=0}^n X_i'$), $n \geq 0$, as in \eqref{eq-treeinvlim} in Section \ref{sec-trees}.
We explain the argument in more detail next.

For $n \geq 0$, denote by $p_n: \widehat X \to V_n$ and $p_n': \widehat X' \to V_n'$ the projections. Then $p_n' \circ \phi: \widehat X \to V_n'$ is a map from $\widehat X$ to a finite space $V_n'$, and by \cite[Lemma 1.1.16]{RZ} for each $n \geq 0$ there is $i_n \geq 0$ and a map $\phi_n: V_{i_n} \to V_n'$ such that $p_n' \circ \phi = \phi_n \circ p_n$. Since $\phi$ is equivariant with respect to the actions of $H$ on $\widehat X$ and $\widehat X'$, every $\phi_n$, $n \geq 0$, is equivariant with respect to the actions of $H$ on $V_{i_n}$ and $V_n'$. It is then standard to show that ${\displaystyle \phi = \lim_{\longleftarrow}\{\phi_n: V_{i_n} \to V_n'\}}$, that is, $\phi$ is induced by a tree map.

We now prove that the sequence $m'=(m_1',m_2',\ldots)$ cannot be constant or even bounded.

\begin{prop}\label{p:limsupmn}
 We have ${\displaystyle \limsup_{n \to \infty} m_n'=\infty}$.
\end{prop}

\begin{proof}
We have $E(H)\cong \prod_{n\geq 0} A_n$, so $\prod_{n\geq 0} A_n$ also acts on $\mathcal{T}'$ by rooted tree automorphisms. As before, $e_n$ denotes the identity permutation in $A_n$. For $n\geq 1$, let
\[
Y_n=\{e_0\}\times\cdots\times\{e_{n-1}\}\times A_n\times\{e_{n+1}\}\times\cdots \subset E(H),
\]
this subgroup is clearly isomorphic to $A_n$. For any $\widehat x \in \widehat X$, the orbit $Y_n \cdot \widehat{x}$ has the same cardinality as $X_n$ and, moreover, the action of $Y_n$ on $Y_n \cdot \widehat x$ is conjugate to the action of $A_n$ on $X_n$. Thus the action of $Y_n$ on the orbit $Y_n \cdot \widehat{x}$ is primitive, and $V_n$ is the union of a finite number of orbits of $Y_n$ of the same cardinality $|X_n|$. Since $Y_n$ acts trivially on $X_k$ for $k \geq n+1$, then for each $k \geq n$ the vertex set $V_k$ is the union of a finite number of orbits of $Y_n$, each of cardinality $|X_n|$. 

For each $k \geq 0$ let $\phi_k: V_{i_k} \to V_k'$ be the mapping defined before the proposition, and consider the partition of $V_k'$ into orbits of $Y_n$. Since the action of $A_n$ on $X_n$ is primitive, by Lemma \ref{cor-trivquotient} each $Y_n$-orbit in $V_k'$ is either a singleton, or a set of cardinality $|X_n|$. 
Let $ \ell  \geq 0$ be the largest index such that the partition of $V_\ell'$ into orbits of $Y_n$ is a partition into singletons. Then the projection $\pr_{\ell}: V_{\ell+1}' \to V_\ell'$ has fibers of cardinality at least $|X_n|$.  Since ${\displaystyle \lim_{n\to \infty} |X_n|=\infty}$, the statement of the proposition follows.
\end{proof}

\begin{rem}\label{rem-assouad}
A Cantor set $X$ with an ultrametric $D$ admits a bi-Lipschitz embedding into a Euclidean space if and only if its Assouad dimension is finite \cite[Sec.~3.2]{Belissard}. For the boundary $\widehat X$ of a spherically homogeneous  tree $\tree$ with ultrametric $D$ defined for instance by \eqref{eq-metricd}, it is easy to show that if $(\widehat{X},D)$ has finite Assouad dimension, then $\tree$ must have bounded spherical index $m = (m_1,m_2,\ldots)$; that is, there must exist $M >1$ such that for all $\ell \geq 1$ we have $m_\ell \leq M$. Proposition~\ref{p:limsupmn} shows that for every ultrametric $D$ on the Cantor set $\widehat{X}$ in Theorem \ref{thm-main1}, which is invariant under the action of $H$, $(\widehat X,D)$ has infinite Assouad dimension.
\end{rem}

\begin{rem}\label{rem-zhang}
The argument in Proposition \ref{p:limsupmn} is valid for any choice of functions $p$ and $o$ satisfying~\eqref{po} in Theorem \ref{thm-main1}. We remark that it is possible to choose $o(n)$ so that it takes values on distinct primes. By the celebrated result by Zhang~\cite{Zhang} there is a constant $C>0$ so that there are infinitely many consecutive primes $p <p '$ satisfying $p '<p +C$. Then we define $p(1,n)$, $p(2,n)$ and $o(n)$ by induction as follows: let $p(1,1)$ be the first prime greater than $C$, and let $p(2,1)$ and $o(1)$ be consecutive primes such that $o(1)<p(2,1)+C<p(2,1)+p(1,1)$. For $n>1$, let $p(1,n)$ be a prime greater than $p(1,n-1)$, $p(2,n-1)$ and $o(n-1)$, and let $p(2,n)$ and $o(n)$ be consecutive primes such that $o(n)<p(2,n)+C<p(2,n)+p(1,n)$.

If $o(n)$ is a prime for $n \geq 1$, then the discussion above Problem \ref{prob-constant} shows that the action in Theorem \ref{thm-main1} is not conjugate to an action by root-preserving automorphisms of a tree with constant spherical index. 
\end{rem}

\section{Wild actions of wreath product groups}

In this section we prove Theorems \ref{thm-main2} and \ref{thm-main3}.

\subsection{Wreath product groups}\label{subsec-wreathproduct} We recall the definition of the wreath product, see \cite{BOERT1996} for details.

Let $G$ and $H$ be finite groups acting by permutations on finite sets $X$ and $Y$ respectively. Denote by $f: X \to H$ a function which assigns a permutation of $Y$ to each $x \in X$, and let
  $H^{|X|} = \{f: X \to H\}$ be the set of all such functions. Note that $G$ acts on $H^{|X|}$ by shifting the argument, that is,
  $$G \times H^{|X|} \to H^{|X|}: (s,f) \mapsto f^s,$$
  where 
  $$f^s(x) = f(s^{-1}\cdot x).$$
  Then the \emph{wreath product} $G \ltimes H^{|X|} $
  is a group with group product $$(s_1,f_1)\circ (s_2,f_2)=(s_1 s_2, f_1 f_2^{s_1}),$$ acting on $X \times Y$ by
    \begin{eqnarray}\label{wp-action} (s,f)\cdot (x,y) = (s\cdot x, f(s\cdot x)\cdot y). \end{eqnarray}
That is, the action \eqref{wp-action} permutes the copies of $Y$ in the product $X\times Y$, while permuting elements within each copy of $Y$ independently by permutations defined by the function $f$.

For convenience of the reader we also spell out the action of the product of two elements on $(x,y) \in X \times Y$, as we will use this formula often in our proofs. Namely, we have
\begin{align*}(s_1,f_1)  \circ (s_2,f_2) \cdot(x,y) & = (s_1s_2,f_1f_2^{s_1}) \cdot (x,y) = (s_1s_2\cdot x,f_1(s_1s_2 \cdot x)f_2(s_2\cdot x) \cdot y),\end{align*}
where $s_1s_2$ and $f_1(s_1s_2\cdot x)f_2(s_2 \cdot x)$ are compositions of permutations of $X$ and $Y$ respectively. In our notation, we reserve $\circ$ for the composition of two elements in the semi-direct product $G \rtimes H^{|X|}$, concatenation for the composition of permutations in the groups $G$ and $H$, and $\cdot$ for the actions of the groups $G$ and $H$ on the sets $X$ and $Y$ respectively.

Now let $\tree$ be a spherically homogeneous tree with spherical index $m = (m_1,m_2,\ldots)$, that is, its set of vertices is $V = \bigsqcup_{n \geq 0} V_n$, where $V_0 = X_0$ is a singleton, $V_n = \prod_{i=0}^n X_i$, and for $i \geq 1$ the set $X_i$ is a set with $m_i$ elements. Then boundary of $\tree$ is the set $\widehat{X} \cong \prod_{i \geq 0} X_i$. For $n \geq 1$, let $\tree(n)$ be a finite subtree of $\tree$ with the vertex set $\bigsqcup_{i = 0}^n V_i$. Then every root-preserving automorphism $a \in \Aut(\tree)$ restricts to a root-preserving automorphism of $\tree(n)$. 

Recall that we denote by $\Sym(m_i)$ the symmetric group on $m_i$ symbols. Then $\Aut(\tree(1)) \cong \Sym(m_1)$. For $\tree(2)$ the automorphism group $\Aut(\tree(2))$ is smaller than $\Sym(m_1 m_2)$, since every automorphism must preserve connected paths in $\tree(2)$ consisting of two edges and so 
  $$\Aut(\tree(2))  \cong \Sym(m_1) \ltimes \Sym(m_2)^{|V_1|} =  \Sym(m_1) \ltimes \Sym(m_2)^{|X_1|}$$
is defined by \eqref{wp-action}. Continuing by induction, one obtains that $\Aut(\tree(n))$ is isomorphic to the finite wreath product
  \begin{align}\label{eq-projAut}\Aut(\tree(n)) = \Aut(\tree(n-1)) \ltimes \Sym(m_n)^{|V_{n-1}|} = \Sym(m_1) \ltimes \Sym(m_2)^{|V_1|} \ltimes \cdots \ltimes \Sym(m_n)^{|V_{n-1}|},\end{align}
  and in the limit $\Aut(\tree)$ is isomorphic to the infinite wreath product of the symmetric groups $\Sym(m_i)$, which we denote $\Sym(m)^\infty$. 
  
For $i \geq 1$, let $A_i \subset \Sym(m_i)$ be a subgroup which acts transitively on the set $X_i$. Then similarly to the above construction one can define the infinite wreath product group
   $$A^\infty: = A_1 \ltimes A_2^{|V_1|} \ltimes A_3^{|V_2|} \ltimes \cdots \subset \Sym(m)^\infty,$$
which also acts on the boundary $\widehat X$ of $\tree$.
  
\begin{rem} \label{rem-countablegen}
Taking the inverse limit of the projections on the first factor in \eqref{eq-projAut}, we can write $\Aut(\tree)$ as the inverse limit of finite groups
  $${\displaystyle \Aut(\tree) = \lim_{\longleftarrow}\{\pr_{n-1}:\Aut(\tree(n)) \to \Aut(\tree(n-1))\}},$$
indexed by natural numbers. Then by \cite[Proposition 4.1.3]{Wilson} $\Aut(\tree)$ is topologically countably generated, that is, it contains a countably generated dense subgroup $G$. With the discrete topology, $G$ is a countably generated infinite discrete group.
\end{rem} 
 
\subsection{Actions of wreath products are algebraically stable}\label{subsec-alternating} We now prove Theorem \ref{thm-main2}. Namely, we show that, for a minimal action of a countably generated group $G$, if the Ellis group $E(G)$ is isomorphic to the wreath product of finite groups, then the action of $G$ (and of $\widehat G$) is dynamically wild and algebraically stable with trivial direct limit centralizer group. 

Let $\tree$ be a spherically homogeneous tree with spherical index $m = (m_1, m_2,\ldots)$, and let $\widehat G = A_1 \ltimes A_2^{|V_1|} \ltimes \cdots$ be the infinite wreath product of the finite groups, acting on the boundary $\widehat X$ of $\tree$, defined in Section \ref{subsec-wreathproduct}. By assumption for $i \geq 1$ the action of $A_i$ on $X_i$ is transitive, then for each $n \geq 1$ the finite wreath product $H_n = A_1 \ltimes A_2^{|V_1|} \ltimes \cdots \ltimes A_n^{|V_{n-1}|}$ acts transitively on $V_n$. By an argument similar to the one in \cite[Proposition~4.1.3]{Wilson}, $\widehat G$ contains a countably generated dense subgroup $G$. For any $g \in G$, the restriction $g|V_n$ is an element of $H_n$, and since $G$ is dense in $\widehat G$, for any $h \in H_n$ there is $g \in G$ such that $g|V_n = h$. Thus $G$ acts transitively on every $V_n$, $n \geq 0$, and it follows that the action of $G$ on $\widehat X$ is minimal.

By \cite[Theorem 1.10]{Lukina2019b} the infinite wreath product $\widehat G$ contains non-Hausdorff elements, and so by \cite[Theorem 1.7]{HL2019b} the action $(\widehat X, G)$ is dynamically wild not of finite type. We now show that the centralizer direct limit group of the action $(\widehat X, G)$ is trivial by an explicit computation.

Choose an infinite path $\widehat{x} = (x_0,x_1,x_2,\ldots) \in \widehat X$, where $x_n \in V_n$ for $n \geq 0$. For $n \geq 0$, define a decreasing sequence of clopen neighborhoods of $\widehat x$ by
\begin{align}\label{eq-nbhdzero}U_n = \{ \widehat w = (w_0, w_1,w_2, \ldots) \in \widehat{X} \mid w_n= x_n\,\}, \end{align}
that is, $U_n$ consists of all paths containing the vertex $x_n$. In particular, $U_0  = \widehat X$. The action of $g \in G$ preserves $U_n$ if and only if $g$ fixes $x_n$. For $n \geq 0$ denote by
    \begin{align}\label{eq-gn}G_n = \{g \in G \mid g \cdot x_n = x_n \}\end{align}
the isotropy group of the action of $G$ at $U_n$. Then the restricted action of $G$ to $U_n$ is that of $G_n$. Denote by $\cD_{\widehat x}$ the discriminant group of the action, defined in Section \ref{subsec-EG}, and for $n \geq 0$ let
  $K_n$ and $Z_n$ be the stabilizer and the centralizer subgroups defined by \eqref{eq-Kchain} and \eqref{eq-Zchain} for $E(G)_{\widehat x} \cong \cD_{\widehat x}$ respectively.  
  Since by assumption the action of $G$ is effective, $K_0$ and $Z_0$ are trivial.

For $i > n$, we say that $w_i \in V_i$ is a \emph{descendant} of $x_n \in V_n$, if $w_i$ is connected to $x_n$ by a finite path of edges without self-intersections joining the vertices in $\bigsqcup_{k=n}^i V_k$ in $\tree$. For $n \geq 0$, let $\tree_n$ be an infinite subtree of $\tree$ with the vertex set $V(\tree_n) = \bigsqcup_{i \geq n} V_i^n$, where $V_i^n$ are defined as follows for $i \geq n$ (see Figure \ref{fig:Tn}):
\begin{figure}
\includegraphics[width=0.60\textwidth]{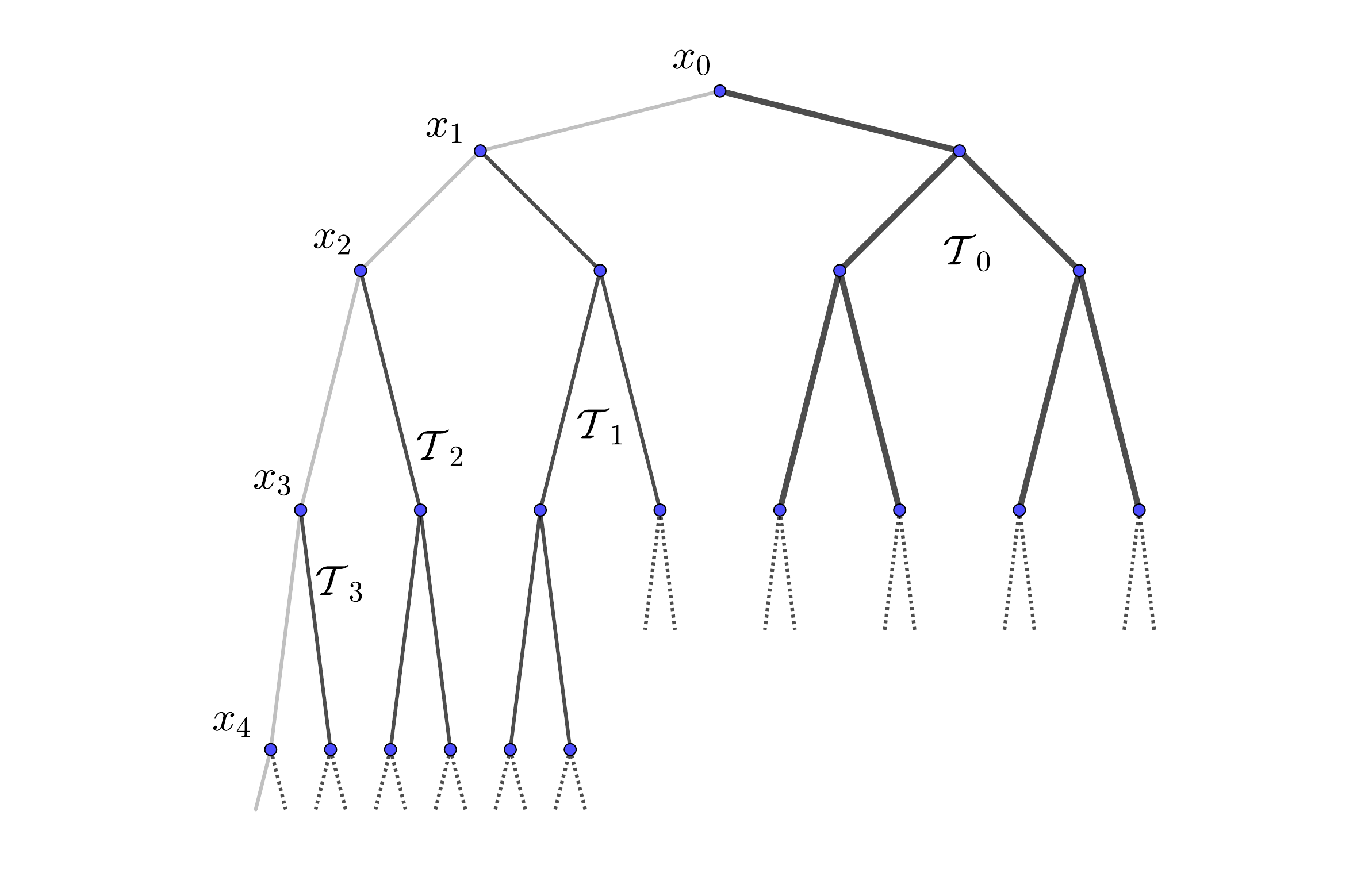}
\caption{Subtrees $\tree_n$, for $n=0,1,2,3$, of a tree $\tree$ with constant spherical index $(2,2,\ldots)$}
\label{fig:Tn}
\end{figure}
\begin{enumerate}
\item The set $V^n_n = \{x_n\}$ is a singleton,
\item For $i > n$, 
  $$V^n_i = \{w_i \in V_i \mid w_i \textrm{ is a descendant of }x_n \textrm { and }w_i \textrm{ is not a descendant of }x_{n+1}\}.$$ 
\end{enumerate}
Then the set $V = V(\tree)$ of vertices of $\tree$ is the disjoint union $\bigsqcup_{n \geq 0} V(\tree_n)$. For $n \geq 0$ denote by 
  $$\widehat X_n = \{(w_i) \in \widehat X \mid w_i = x_i \textrm{ for } 0 \leq i < n, \textrm{ and } w_i \in V^n_i \textrm{ for all }i \geq n\}$$
the set of paths contained in the tree $\tree_n$ on levels $V_i$, $i \geq n$, and note that $\widehat{X}_n = U_n \backslash U_{n+1}$. The tree $\tree_n$ has $x_n$ as a root, while paths in $\widehat X_n$ start at $x_0 \in V_0$, so although $\widehat X_n$ is isomorphic to the boundary of the tree $\tree_n$, $\widehat X_n$ and the boundary of $\tree_n$ are not the same object.

\begin{lem}\label{lem-invtn}
\begin{enumerate}
\item The set $\widehat X_n$ is invariant under the action of $\cD_{\widehat{x}}$, for $n \geq 0$.
\item Let $n \geq 0$ and let $\widehat g \in K_n$. Then $\widehat g|\widehat X_i = \id$ for $i \geq n$.
\end{enumerate}
\end{lem}

\begin{proof}
Since the action of every $\widehat g \in \cD_{\widehat{x}}$ fixes $x_n$ and $x_{n+1}$ for $n \geq 0$, we have that the action of $\widehat g$ preserves both $U_n$ and $U_{n+1}$. Therefore, the action of $\widehat g$ preserves $\widehat X_n$ and we proved item $(1)$. Then item $(2)$ follows from the definition of $K_n$ and the fact that for $i \geq n$ we have $\widehat X_i \subset U_n$.
\end{proof}

Thus by Lemma \ref{lem-invtn} for every $n \geq 0$ the restriction of the action of $\widehat g \in \cD_{\widehat x}$ to $\tree_n$ defines a homomorphism 
  $$\cD_{\widehat x} \to \Aut (\tree_n): \widehat{g} \mapsto \widehat{g}|\tree_n.$$ 
Denote by $S_n$ the image of this homomorphism in $\Aut(\tree_n)$. Then by a direct computation one obtains that
\begin{equation}\label{eq:infsp}
S_n \cong (A_{n+1})_{x_{n+1}} \ltimes A_{n+2}^{|V^n_{n+1}|} \ltimes A_{n+3}^{|V^n_{n+2}|} \ltimes \dots,    
\end{equation}
where $(A_{n+1})_{x_{n+1}} = \{\tau \in A_{n+1} \mid \tau \cdot x_{n+1} = x_{n+1}\}$. That is, the restriction of $S_n$ to $V_{n+1} \cap \tree_n$ consists of permutations in $A_{n+1}$ which fix $x_{n+1}$. For $i \geq n+1$ the values of the functions $f: V_i \to A_{i+1}$ at different $v \in V_i$ are independent, which by an argument similar to the one in Section \ref{subsec-wreathproduct} yields $\eqref{eq:infsp}$.

\begin{lem}\label{lem-lemm34}
\begin{enumerate}
\item For $n \geq 1$ we have $K_n \cong \prod_{i=0}^{n-1} S_{i}$. 
\item For any $n \geq 1$ we have $|K_n| = \infty$.
\item $\cD_{\widehat x} \cong \prod_{i \geq 0} S_{i}$.
\end{enumerate}
\end{lem}

\begin{proof}
Let $\X_n = \bigcup_{i=0}^n \widehat X_i$. Since the trees $\tree_i$ for $0 \leq i \leq n$ are disjoint and invariant under the action of $\cD_{\widehat x}$, the product $\prod_{i=0}^{n} S_{i}$ acts on $\X_n$. For $i \geq 0$ the values of the functions $f: V_i \to A_{i+1}$ in the wreath product at different $v \in V_i$ are independent, and they take all values in $A_{i+1}$, therefore for $n \geq 0$ the map 
  \begin{equation}\label{eq-Dxproduct} \cD_{\widehat x} \to \prod_{i=0}^{n} S_{i} \colon \widehat{g} \mapsto (\widehat{g}|\tree_0, \widehat{g}|\tree_1, \ldots, \widehat{g}|\tree_n) 
  \end{equation}
is surjective. Note that if the actions of $\widehat h, \widehat g \in \cD_{\widehat x}$ differ on some $V_i$ for $0 \leq i \leq n$, then their images under \eqref{eq-Dxproduct} are distinct. By item $(2)$ in Lemma \ref{lem-invtn} for each $\widehat g \in K_n$ we have $\widehat g|\widehat X_i = \id$ for $i \geq n$, so it follows that the restriction of the map \eqref{eq-Dxproduct} to $K_n$ is injective and $(1)$ holds. 

The group $K_0$ is trivial since the action of $G$ on $\widehat X$ is effective. By \eqref{eq:infsp} the groups $S_n$ are infinite, then $(1)$ implies that $K_n$ are infinite for all $n \geq 1$, and $(2)$ is proved.  Statement $(3)$ is straightforward.
\end{proof}

Statement $(2)$ of Lemma \ref{lem-lemm34} shows that the action of a dense subgroup of a wreath product group is wild not of finite type, providing an alternative proof of \cite[Theorem 7.6(1)]{HL2019b} for wreath products.

Now we can show that the centralizer groups $Z_n$ are trivial for $n \geq 0$.

\begin{lem}
$Z_{n}$ is trivial for any $n \geq 0$.
\end{lem}

\begin{proof}
First note that since $Z_0 \subset K_0$ and $K_0$ is trivial, then $Z_0$ is trivial.

Recall that for any $i \geq 1$ we have $H_i = A_1 \ltimes A_2^{|V_1|} \ltimes \cdots \ltimes A_i^{|V_{i-1}|}$, and for any $\widehat h \in \widehat G$ we can write for its restriction to $V_{i+1}$
   $$\widehat h|V_{i+1} = (h_1,h_2) \in H_i \ltimes A_{i+1}^{|V_{i}|} \cong H_{i+1},$$
where $h_{1} \in H_i$ and $h_{2} : V_i \to A_{i+1}$ is a function.

In particular, for $\widehat g \in Z_n$ we have $\widehat g|V_{i+1} = (g_1,g_2)$ for $i \geq 1$. Let $\widehat h \in \widehat G$ then we have $\widehat h|V_{i+1} = (h_1,h_2)$. Suppose $i \geq n$ and $h_1\cdot x_i = x_i$. Then $h_1\cdot x_n = x_n$, and so $\widehat h \in \widehat U_n$, where $\widehat U_n$ is such that $\widehat U_n/\cD_{\widehat x} = U_n$ for $U_n$ defined in \eqref{eq-nbhdzero}, and such that $G_n = G \cap \widehat U_n$. Note that for any $\widehat h' \in \widehat G$, if $\widehat h'|V_i = h_1$, then $\widehat h' \in \widehat U_n$ since $h_1$ fixes $x_i$. In other words, for any function $h_2' : V_i \to A_{i+1}$, any extension of $(h_1,h_2')$ from $V_{i+1}$ to $\tree$ is in $\widehat U_n$.

We are going to show that $g_1$ is the trivial permutation of $V_i$ by carefully choosing various $h_2$ and using the fact that $\widehat g$ and $\widehat h$ commute. In the arguments below, $g_1,g_2$ and $h_1$ are fixed. For a fixed $h_1$, we choose different functions $h_2: V_i \to A_{i+1}$. Our goal is to show that for some choice of $v \in V_i$ and some choice of $h_2$ depending on $v$ we have $g_1\cdot v = v$. Since $g_1$ is fixed, and $g_1\cdot v = v$ holds for some choice of $h_2$, this statement must hold for any choice of $h_2$. Since the argument is for arbitrary $v \in V_i$, this will show that $g_1\cdot v = v$ for any $v \in V_i$, which is what we want to prove. We assume that $i \geq n$ and $h_1\cdot x_i = x_i$, which means that the element $(h_1,h_2)$ is the restriction of an element of $\widehat U_i \subset \widehat U_n$. By assumption $(g_1,g_2) \in Z_n$, so the elements $(h_1,h_2)$ and $(g_1,g_2)$ of the group $H_i \rtimes A_{i+1}^{|V_i|}$ commute by the definition of $Z_n$.
 
So first take $h_2 = \widetilde{e}: V_i \to A_{i+1}$ be the trivial function, that is, for any $v \in V_i$ we have $\widetilde{e}(v) = e_{i+1}$, where $e_{i+1}$ is the identity in $A_{i+1}$. Since $\widehat g \in Z_n$, and $i \geq n$, $\widehat{g}$ commutes with $\widehat h$, and we have
  $$(g_1,g_2)\circ(h_1,\widetilde e) = \widehat{g} \,\widehat{h}|V_{i+1}=\widehat{h} \, \widehat{g}|V_{i+1} = (h_1,\widetilde e)\circ(g_1,g_2).$$
Applying to this formula the definition of the wreath product, we obtain that for any $(v,k) \in V_i \times X_{i+1}$
 \begin{align*}
  (h_1,\widetilde e) \circ (g_{1},g_{2}) \cdot (v,k)& = (h_1 g_{1}\cdot v, g_{2} (g_1\cdot v) \cdot k), \\   (g_{1},g_{2})\circ (h_1,\widetilde e) \cdot (v,k)& = (g_{1}h_{1}\cdot v, g_{2}( g_1h_1\cdot v) \cdot k ), 
\end{align*}
and so we have $g_1h_1 = h_1g_1$ and for each $v \in V_i$
  \begin{align}\label{eq-trivh}g_{2} ( g_1h_{1} \cdot v) = g_{2} ( g_1 \cdot v).\end{align}
Now assume that $h_2$ is arbitrary. By definition of $Z_n$ we have that $\widehat g$ commutes with $\widehat h$, so
$$(g_1,g_2)\circ (h_1, h_2) = \widehat{g} \,\widehat{h}|V_{i+1}=\widehat{h} \, \widehat{g}|V_{i+1} = (h_1,h_2)\circ (g_1,g_2).$$
 Using the definition of the wreath product, we can write for $(v,k) \in V_i \times X_{i+1}$
\begin{align}\label{eq-f2}
    (g_1,g_2) \circ (h_1,h_2) \cdot (v,k)& = (g_1 h_1\cdot v, g_2 ( g_1 h_1 \cdot v)  h_2 ( h_1 \cdot v) \cdot k), \end{align}
where  $g_2 (g_1 h_1 \cdot v)$ and $h_2( h_1 \cdot v)$ are permutations in $A_{i+1}$, and the concatenation denotes their product. Similarly,
\begin{align} \label{eq-f3}
    (h_1,h_2) \circ (g_1,g_2) \cdot (v,k)& = (h_1 g_1 \cdot v , h_2 ( h_1 g_1\cdot v) g_{2} ( g_1 \cdot v) \cdot k).
\end{align}
Using \eqref{eq-trivh} we obtain from \eqref{eq-f2}
  $$g_2 (g_1 h_1 \cdot v) h_2 ( h_1 \cdot v)  = g_2 ( g_1 \cdot v)h_2 ( h_1\cdot v),$$
 and combining that with \eqref{eq-f3} we must have for each $v \in V_i$
  $$g_2 ( g_1 \cdot v)h_2 (h_1 \cdot v) =  h_2 ( h_1 g_1\cdot v) g_{2} ( g_1 \cdot v).$$
Therefore, for each $v \in V_i$ we have
  \begin{align}\label{eq-f55}  h_2 ( h_1 g_1 \cdot v)  = g_2 ( g_1 \cdot v) h_2 ( h_1\cdot v) g_{2} ( g_1 \cdot v)^{-1}. \end{align}
We will show that for a fixed $v \in V_i$ we have $g_1\cdot v =v$. Since $v \in V_i$ is arbitrary and $g_1$ is fixed, this will prove that $g_1$ is the trivial permutation of $V_i$.

Fix $v \in V_i$. Let $b \ne e_{i+1} \in A_{i+1}$ and define $h_{2}: V_i \to A_{i+1}$ by $h_{2} ( h_1\cdot v) = b$, and $h_{2} ( h_1\cdot u) = e_{i+1}$ for all $u \ne v \in V_i$. This is possible even under the assumption that $h_1 \cdot x_i = x_i$. If $h_1 g_1 \cdot v \ne h_1\cdot v$, then by \eqref{eq-f55}
\[
g_2 ( g_1\cdot v) \, b \,g_2 ( g_1\cdot v)^{-1} =g_2 ( g_1\cdot v) h_2 ( h_1\cdot v)g_2 (g_1\cdot v)^{-1}  = h_2 ( h_1g_1 \cdot v) =e_{i+1}
\]
which implies that $b = e_{i+1}$, and which contradicts the choice of $b$. Therefore, we must have $h_1 g_1\cdot v = h_1\cdot v $. Since $h_1$ is bijective, this implies that $g_1\cdot v = v$. Repeating this argument for any $v \in V_i$ we obtain that $g_1$ must be the trivial permutation of $V_i$.

Note that if $g_1$ is the identity permutation of $V_i$, then the restriction $\widehat{g}|V_k = g_1|V_k$ is the identity permutation for all $0 \leq k \leq i$. If $\widehat g \in Z_n$ is non-trivial, then for some $i \geq n$ we must have that $\widehat{g}|V_i = g_1$ is non-trivial, which contradicts the argument in the lemma. It follows that $Z_n$ is the trivial group. 
\end{proof}

 This completes the proof of Theorem \ref{thm-main2}.

\section{Dynamically and algebraically wild actions}

In this section we prove Theorem \ref{thm-main3} by constructing a family of actions which are dynamically and algebraically wild. 
Let $H$ and $G$ be countable groups acting minimally by root-preserving automorphisms on the trees $\tree_H$ and $\tree_G$, with respective boundaries $\widehat X = \prod_{n\geq 0} X_n$ and $ \widehat Y = \prod_{n\geq 0}Y_n$. 
Denote by $\widehat H \subset \Aut(\tree_H)$ and $\widehat G \subset \Aut(\tree_G)$ the corresponding Ellis groups.
Fix a sequence $\widehat x = (x_n)_{n\geq 0} \in \widehat X$ and, for $n \geq 0$, let 
  $$U_n = \{ \widehat w = (w_n)_{n \geq 0} \mid w_i = x_i \textrm{ for }0 \leq i \leq n\}$$
be a clopen neighborhood of $\widehat x$ in $\widehat X$. Denote by $K_n^H$ and $Z_n^H$ the corresponding stabilizer and centralizer subgroups of the discriminant group $\cD_{\widehat{x}} = \widehat H_{\widehat x}$.

Similarly, fix a sequence $(y_n)_{n\in\N} \in \widehat Y$, and for $n \geq 0$ let 
  $$V_n = \{ \widehat z =(z_n)_{n \geq 0} \mid z_i = y_i \textrm{ for }0 \leq i \leq n\}$$
be a clopen neighborhood of $\widehat y$ in $\widehat Y$. Denote by $K_n^G$ and $Z_n^G$ the corresponding stabilizer and centralizer subgroups of the discriminant group $ \widehat G_{\widehat x}$.

Next, consider the product space $\widehat X \times \widehat Y = \prod_{n \geq 0} X_n\times Y_n$, and the product action of the group $H\times G$ on this space defined by
\[
(H \times G) \times (\widehat X \times \widehat Y ) \to (\widehat X \times \widehat Y ): ((\widehat h, \widehat g),(\widehat w,\widehat z))\mapsto (\widehat h\cdot \widehat w,\widehat g \cdot \widehat z).
\]
Denote by $E(H \times G)$ the closure of the action in $\operatorname{Homeo}(\widehat X \times \widehat Y)$, then $E(H \times G) = \widehat H \times \widehat G$.
Denote by $K_n$ and $Z_n$ the sequence of stabilizer and centralizer subgroups for the product action with respect to the sequence of clopen neighborhoods $U_n\times V_n$, for $n \geq 0$.

Now suppose that the action of $H$ on $\widehat X$ is wild of flat type, that is, the sequence $\{K^H_n\}_{n \geq 0}$ is unbounded and $K_n^H=Z_n^H$ for all $n \geq 0$. Such actions are described, for instance, in Example \ref{thm-firstexamples}, item 2. Suppose that the action of $G$ on $\widehat Y$ is dynamically wild with trivial centralizer group; more precisely, the sequence $\{K_n^G\}_{n \geq 0}$ is unbounded and $Z_n^G$ is trivial for all $n \geq 0$. Such actions are described in Theorems \ref{thm-main1} and \ref{thm-main2}. 
It is straightforward that for the product action of $H \times G$ on $\widehat X \times \widehat Y$ we have for $n \geq 0$
\begin{equation}\label{eq-KnZn}
K_n=K_n^H\times K_n^G,\qquad Z_n=Z_n^H\times Z_n^G = Z_n^H \times \{e\},
\end{equation}
where $e$ is the identity in $\widehat G$. 
Since the action of $H$ on $\widehat X$ is wild of flat type, the group chains $\{K_n^H\}_{n \geq 0}$ and $\{Z_n^H\}_{n \geq 0}$ are both unbounded. It follows that the group chains $\{K_n\}_{n \geq 0}$ and $\{Z_n\}_{n\geq 0}$ are unbounded as well, and so the product action is wild and algebraically wild. If, in addition, $(X,H,\Phi)$ is algebraically wild of finite type, then $Z_n \cong Z_n^H$ is a finite group for $n \geq 0$, and the product action is algebraically wild of algebraic finite type. 

Since $K_n^G$ is non-trivial for any $n \geq 0$, from \eqref{eq-KnZn} it follows that the inclusion of direct limit groups $\U_c^x \subset \U_s^x$ is proper, and the action of $H \times G$ on $\widehat X \times \widehat Y$ is dynamically wild.


\begin{thebibliography}{10}

  \bibitem{ALC2009}
{J.~{\'A}lvarez L{\'o}pez and A.~Candel},
\newblock {\it Equicontinuous foliated spaces},
\newblock {\bf Math. Z.}, 263:725--774, 2009.

\bibitem{ALM2016}
{J.~{\'A}lvarez L{\'o}pez and M.~Moreira Galicia},
\newblock {\it Topological {M}olino's theory},
\newblock {\bf Pacific. J. Math.}, 280:257--314, 2016.

\bibitem{AB2019}
 {J.~\'{A}lvarez L\'{o}pez and R.~Barral Lij\'{o}},
 {\it Molino's description and foliated homogeneity},
 {\bf Topology Appl.}, {260}:148--177, {2019}.

\bibitem{Auslander1988}
{J.~Auslander},
\newblock {\bf Minimal flows and their extensions},
\newblock {North-Holland Mathematics Studies}, Vol. 153, {North-Holland Publishing Co., Amsterdam}, 1988.


\bibitem{BN2008}
{L.~Bartholdi and V.~Nekrashevych}
\newblock{\it Iterated monodromy groups of quadratic polynomials, I},
\newblock{\bf Groups Geom. Dyn.}, 2:309-336, 2008.

\bibitem{BOERT1996}
{H.~Bass, M.~V.~Otero-Espinar, D.~Rockmore, C.~Tresser},
\newblock {\bf Cyclic Renormalization and Automorphism Groups of Rooted Trees},
\newblock LNM 1621, Springer 1996.

\bibitem{Belissard}
{J.V.~Belissard and A.~Julien}
\newblock{\it Bi-Lipshitz embedding of ultrametric Cantor sets into Euclidean spaces}, arXiv:1202.4330,
 2012.

\bibitem{Bondarenko2010}
{I.~Bondarenko},
  {\it Finite generation of iterated wreath products},
   {\bf Arch. Math. (Basel)}, {95(4)}: 301--308, 2010.

\bibitem{BHL2017}
{M.~Bush, W.~Hindes and N.~Looper},
     {\it Galois groups of iterates of some unicritical polynomials},
   {\bf Acta Arith.}, {181(1)}:57--73, 2017.
   
\bibitem{DHL2016}
{J.~Dyer, S.~Hurder and O.~Lukina},
\newblock {\it The discriminant invariant of Cantor group actions},
\newblock {\bf Topology Appl.}, 208: 64-92, 2016.
   

\bibitem{DHL2017}
{J.~Dyer, S.~Hurder and O.~Lukina},
\newblock {\it Molino theory for matchbox manifolds},
\newblock {\bf Pacific J. Math.}, 289:91-151, 2017.

\bibitem{Ellis1960}
{R.~Ellis},
\newblock {\it A semigroup associated with a transformation group},
\newblock {\bf Trans. Amer. Math. Soc.}, 94:272--281, 1969.


 \bibitem{EllisGottschalk1960}
{R.~Ellis and W.H.~Gottschalk},
\newblock {\it Homomorphisms of transformation groups},
\newblock {\bf Trans. Amer. Math. Soc.}, 94:258--271, 1969.

\bibitem{CC2000}
{A.~Candel and L.~Conlon},
 {\bf Foliations. {I}}, {Graduate Studies in Mathematics}, {23}, {2000}, {xiv+402}pp.

\bibitem{ClarkHurder2013}
{A.~Clark and S.~Hurder},
\newblock {\it Homogeneous matchbox manifolds},
\newblock {\bf Trans. A.M.S.}, 365:3151--3191, 2013.

\bibitem{ClarkHurder2011}
{A.~Clark and S.~Hurder},
  {\it Embedding solenoids in foliations},
   {\bf Topology Appl.}, 158(11):1249--1270, 2011.


\bibitem{FO2002}
{R.~Fokkink and L.~Oversteegen},
\newblock {\it Homogeneous weak solenoids},
\newblock {\bf Trans.  Amer. Math. Soc.}, 354(9):3743--3755, 2002.

\bibitem{Grig2011}
{R.~Grigorchuk},
{\it Some Topics in the Dynamics of Group Actions on Rooted Trees},
{\bf Proc. Steklov Institute of Math.}, 273: 64--175, 2011.

\bibitem{Hurder2017}
{S.~Hurder},
 {\it Lipschitz matchbox manifolds},
 in {\bf Geometry, dynamics, and foliations 2013},
 {Adv. Stud. Pure Math.}, {72}:{71--115}, {2017}.

\bibitem{HL2019a}
{S.~Hurder and O.~Lukina},
\newblock {\it Wild solenoids},
\newblock {\bf Trans. Amer. Math. Soc.}, 371:4493-4533, 2019.

\bibitem{HL2019b}
{S.~Hurder and O.~Lukina},
\newblock {\it Limit group invariants for wild Cantor actions},
\newblock to appear in {\bf Ergodic Theory Dynam. Systems}; {arXiv:1904.11072}.

\bibitem{Jones2013}
{R. Jones},
\newblock{\it Galois representations from pre-image trees: an arboreal survey}
\newblock in {\bf Actes de la {C}onf\'erence ``{T}h\'eorie des {N}ombres et {A}pplications''}, 107-136, 2013.

 \bibitem{McCord1965}
{C.~McCord},
\newblock {\it Inverse limit sequences with covering maps},
\newblock {\bf Trans. Amer. Math. Soc.}, 114:197--209, 1965.

\bibitem{Miller1928}
{G.~A.~Miller},
{\it Possible orders of two generators of the alternating and of
              the symmetric group},
  {\bf Trans. Amer. Math. Soc.}, {30(1)}: 24--32, {1928}.

\bibitem{Lukina2019}
{O.~Lukina},
\newblock {\it Arboreal {C}antor actions},
\newblock {\bf Journal of the London Math. Society},  99(3): 678--706, 2019.

\bibitem{Lukina2019b}
{O.~Lukina},
\newblock {\it Galois groups and Cantor actions},
\newblock to appear in {\bf Trans. Amer. Math. Soc.}, arXiv:1809.08475.

\bibitem{Nekrashevych2005}
{V.~Nekrashevych},
\newblock {\bf Self-similar groups},
\newblock {Mathematical Surveys and Monographs}, Vol. 117, 
\newblock {American Mathematical Society, Providence, RI}, 2005.  

\bibitem{RZ}
{L.~Ribes and P.~Zalesskii},
{\bf Profinite Groups}, Springer.

\bibitem{Wieland}
{H.~Wieland},
{\bf Finite Permutation Groups}, Academic Press, 1964.

\bibitem{Wilson} J. S. Wilson, \textbf{Profinite groups}, London Mathematical Society Monographs, New Series, 19, 1998, xii+284pp.

 \bibitem{Winkelnkemper1983}
{E.~Winkelnkemper}, 
\newblock {\it The graph of a foliation}, 
\newblock {\bf Ann. Global Ann. Geo.}, 1:51--75, 1983.

   
\bibitem{Zhang}
{Y.~Zhang},
 {\it Bounded gaps between primes},
  {\bf Ann. of Math.}, {179(3)}: 397--415, {2014}.
 \end{thebibliography}
\end{document}